\documentclass[leqno, 11pt]{amsart}
\usepackage{amsmath,amsthm,amssymb,amscd}
\usepackage[mathscr]{eucal}
\usepackage{graphicx}

\makeatletter
\def\@sect#1#2#3#4#5#6[#7]#8{\ifnum #2>\c@secnumdepth
     \let\@svsec\@empty\else
     \refstepcounter{#1}\edef\@svsec{\csname the#1\endcsname.}\fi
     \@tempskipa #5\relax
      \ifdim \@tempskipa>\z@
        \begingroup #6\relax
          \@hangfrom{\hskip #3\relax\bfseries\@svsec}{\bfseries\mathversion{bold} \interlinepenalty \@M #8\par}%
        \endgroup
       \csname #1mark\endcsname{#7}\addcontentsline
         {toc}{#1}{\ifnum #2>\c@secnumdepth \else
                      \protect\numberline{\csname the#1\endcsname}\fi
                    #7}\else
        \def\@svsechd{#6\hskip #3\relax  
                   \@svsec #8\csname #1mark\endcsname
                      {#7}\addcontentsline
                           {toc}{#1}{\ifnum #2>\c@secnumdepth \else
                             \protect\numberline{\csname the#1\endcsname}\fi
                       #7}}\fi
     \@xsect{#5}
}

\def\subsection{\@startsection{subsection}{2}%
  \z@{.5\linespacing\@plus.7\linespacing}{-.5em}%
  {\normalfont\bfseries\mathversion{bold}}}

\@addtoreset{equation}{section}
\makeatother

\theoremstyle{plain}
\newtheorem{theorem}{Theorem}[section]
\newtheorem{proposition}[theorem]{Proposition}
\newtheorem{corollary}[theorem]{Corollary}
\newtheorem{lemma}[theorem]{Lemma}
\newtheorem{rem}[theorem]{Remark}
\newenvironment{remark}{\begin{rem}\normalfont}{\end{rem}}
\newtheorem{ex}[theorem]{Example}
\newenvironment{example}{\begin{ex}\normalfont}{\end{ex}}

\newcommand{\lt}{\left}
\newcommand{\rt}{\right}
\newcommand{\tbf}[1]{\textbf{#1}}

\newcommand{\rint}{\mathbb{Z}}
\newcommand{\rat}{\mathbb{Q}}

\newcommand{\ringint}[1]{\mathfrak{O}_{#1}}
\newcommand{\prim}[1]{\mathfrak{#1}}
\newcommand{\primes}{\mathscr{P}}
\newcommand{\gp}{\mathscr{X}}

\newcommand{\closure}[1]{\overline{#1}}
\newcommand{\eqcls}[1]{\overline{#1}}
\newcommand{\ellunit}{\mathcal{O}}
\newcommand{\gext}{K}
\newcommand{\dual}[1]{\widehat{#1}}

\newcommand{\iclsgp}{\mathrm{Cl}}

\newcommand{\Ker}{\mathrm{Ker}\,}

\newcommand{\Image}{\mathrm{Im}\,}
\newcommand{\Gal}{\mathrm{Gal}}
\newcommand{\Hom}{{\rm Hom}}

\newcommand{\rank}{\mathrm{rank}}
\newcommand{\res}{\mathrm{res}}

\newcommand{\GL}{\mathrm{GL}}
\newcommand{\PGL}{\mathrm{PGL}}
\newcommand{\Norm}[1]{\mathrm{N}_{#1}}

\newcommand{\ideal}[1]{\mathfrak{#1}}
\newcommand{\ord}{{\rm ord}}

\def\mod{\;\mathrm{mod}\;}
\def\rnum#1{\expandafter{\romannumeral #1}}
\def\Rnum#1{\uppercase\expandafter{\romannumeral #1}}

\input{cyracc.def}
\newfam\cyrfam
\font\tencyr=wncyr10	
\font\sevencyr=wncyr7	
\font\fivecyr=wncyr5	
\def\cyr{\fam\cyrfam\tencyr\cyracc}
\textfont\cyrfam=\tencyr
\scriptfont\cyrfam=\sevencyr
\scriptscriptfont\cyrfam=\fivecyr

\def\vsp{\vspace{-6mm}}

\address{Graduate school of Mathematics, Kyushu University, Fukuoka 812-8581, Japan}
\email{rinrin@math.kyushu-u.ac.jp}

\begin{document}
\thispagestyle{empty}

\topskip=3cm
\begin{center}
\tbf{\Large Elliptic curves related to cyclic cubic extensions}\\
\vspace{8mm}
Rintaro Kozuma
\end{center}
\vspace{6mm}

\noindent
{\normalfont\scshape\bfseries Abstract.}\quad
The aim of this paper is to study certain family of elliptic curves $\{\gp_H\}_H$ defined over a number field $F$
 arising from hyperplane sections of some cubic surface $\gp/F$ associated to a cyclic cubic extension $\gext/F$.
We show that each $\gp_H$ admits a $3$-isogeny $\phi$ over $F$
 and the dual Selmer group $S^{(\dual{\phi})}(\dual{\gp_H}/F)$
 is bounded by a kind of unit/class groups attached to $\gext/F$.
This is proven via certain rational function on the elliptic curve $\gp_H$ with nice property.
We also prove that the Shafarevich-Tate group $\text{\cyr X}(\dual{\gp_H}/\rat)[\dual{\phi}]$
 coincides with a class group of $\gext$ as a special case.

\section{Introduction}
\thispagestyle{empty}

The principal objects we are going to investigate are defined as follows.
Let $\gext/F$ be a cyclic cubic extension over a number field $F$
 with Galois group $G=\langle\sigma\rangle \simeq \rint/3\rint$,
 and let $\{1,\,\alpha_1,\,\alpha_2\}$ be a $F$-basis of $\gext$.
Define a projective variety $\gp$ in the projective space $\mathbb{P}^3$ associated to the extension $\gext/F$
 with the $F$-basis $\{1,\,\alpha_1,\,\alpha_2\}$ of $\gext$ by the homogeneous polynomial
\begin{align}\label{defpoly}
f(x,\,y,\,z,\,w):=(x+y\alpha_1+z\alpha_2)(x+y\alpha_1^{\sigma}+z\alpha_2^{\sigma})
(x+y\alpha_1^{\sigma^2}+z\alpha_2^{\sigma^2})-w^3
\end{align}
 with coefficients in $F$, which is a singular cubic hypersurface and rational over $\gext$.
The elliptic curves we consider are hyperplane sections $\{\gp \cap H\}_H$ of $\gp$
 where $H/F$ varies through all hyperplanes in $\mathbb{P}^3$ satisfying $(\gp \cap H)(F) \ne \emptyset$.
We shall simply denote by $\gp_H$ the section $\gp \cap H$.
For such a cubic curve $\gp_H/F$ we can prove (in Lemma~\ref{singularity}) that
$$\begin{cases}
\text{$\gp_H/F$ is birationally equivalent over $F$ to $\mathbb{P}^1$ if $H/F$ is tangent to $\gp$},\\
\text{$\gp_H/F$ is an elliptic curve if $H/F$ is not tangent to $\gp$},
\end{cases}$$
and it also turns out (in Proposition~\ref{normalization}) that there is a hyperplane $H_0/F$ on which the point
 $\ellunit:=[1,\,0,\,0,\,1] \in \gp(F)$ lies such that the cubic curve $\gp_H$ is isomorphic over $F$
 to $\gp_{H_0}$. Therefore we may assume $\ellunit \in H$ and fix $\ellunit$ as a base point on $\gp_H$
 without loss of generality,
 namely the defining equation of $H$ has the form $a(x-w)+by+cz=0$ with $[a,\,b,\,c] \in \mathbb{P}^2(\ringint{F})$
 where $\ringint{F}$ is the ring of integers of $F$.

Under the situation as above, our interest is the arithmetic nature of the elliptic curves $\{\gp_H\}_H$
 and how that relates to the arithmetic of the cyclic cubic extension $\gext/F$.
Main results are stated as follows.

\begin{theorem}\label{WE}
\quad
For any cyclic cubic extension $\gext/F$ with a fixed $F$-basis $\{1,\,\alpha_1,\,\alpha_2\}$ of $\gext$
 and the Galois group $\Gal(\gext/F)=\langle \sigma \rangle$,
 let $\gp \subset \mathbb{P}^3$ be a cubic hypersurface over $F$ defined by the polynomial \eqref{defpoly}.
For a hyperplane $H:a(x-w)+by+cz=0$ in $\mathbb{P}^3$ with coefficients $[a,\,b,\,c] \in \mathbb{P}^2(\ringint{F})$
 which is not tangent to $\gp$,
 there is an isomorphism $\vartheta:\gp_H \stackrel{\sim}{\to} \mathcal{W}(\gp_H)$ over $F$
 mapping $\ellunit$ to $[0,\,1,\,0]$ with an elliptic curve $\mathcal{W}(\gp_H) \subset \mathbb{P}^2$ over $F$
 given by the Weierstrass form
\begin{align}\label{we}
y^2+a\,\delta(\gp)xy-\Norm{\gext/F}\bigl(\gamma(\gp_H)\bigr)y=x^3
\end{align}
where $\Norm{\gext/F}:\gext^{*} \to F^{*}$ is the norm map
 and the constants $\delta(\gp) \in F^{*},\,\gamma(\gp_H) \in \gext^{*}$ are defined by determinants
$$
\delta(\gp)=
\begin{vmatrix}
1	& \alpha_1		& \alpha_2\\
1	& \alpha_1^{\sigma}	& \alpha_2^{\sigma}\\
1	& \alpha_1^{\sigma^2}	& \alpha_2^{\sigma^2}
\end{vmatrix}
, \quad
\gamma(\gp_H)=
\begin{vmatrix}
1	& \alpha_1		& \alpha_2\\
1	& \alpha_1^{\sigma}	& \alpha_2^{\sigma}\\
a	& b			& c
\end{vmatrix}.
$$
\end{theorem}

Note that the isomorphism class of $\mathcal{W}(\gp_H)$ is independent of the choice of the coordinate
 $(a,\,b,\,c) \in \mathbb{A}^3(\ringint{F})$ and the generator $\sigma$ of $\Gal(\gext/F)$.

We retain the notation, and assume $H$ is not tangent to $\gp$.
By the above theorem, there is the $F$-rational $3$-torsion point $(0,\,0)$ on $\mathcal{W}(\gp_H)$,
 thus the elliptic curve $\gp_H/F$ always admits a $3$-isogeny $\phi:\gp_H \to \dual{\gp_H}$ defined over $F$
 where $\dual{\gp_H}/F$ stands for the dual elliptic curve $\gp_H/\langle \vartheta^{-1}(0,\,0) \rangle$.
Let $\dual{\phi}$ denote its dual isogeny.
For any set $S$ of finite primes of $F$, we denote by $\ringint{L,S}^{*},\,\iclsgp_{L,S}$
 the group of $S$-units and the $S$-ideal class group of a finite extension $L/F$ respectively
(The precise definition will be given in \S\ref{sec_SI}).
We make the general notation:

\begin{center}
\begin{tabular}{ll}
$\eqcls{\Omega}$ &: an equivalence class represented by $\Omega$, or an algebraic closure of $\Omega$.\\
$\langle \varepsilon \rangle$ &: a group generated by an element $\varepsilon$.\\
$\Gamma[\varphi]$ &: the kernel of a homomorphism $\varphi:\Gamma \to \Gamma'$.\\
$\nu_{\prim{p}}$ &: the $\prim{p}$-adic normalized valuation of a local field $L_{\prim{p}}$.
 $\bigl(\text{i.e. } \nu_{\prim{p}}(L_{\prim{p}})=\rint\bigr)$\\
$\mu_3(L)$ &: the group of all cube roots of unity in a field $L$.
\end{tabular}
\end{center}

\noindent
As for the Selmer groups $S^{(\dual{\phi})}(\dual{\gp_H}/F),\,S^{(\phi)}(\gp_H/\rat)$, we have the following results.

\begin{theorem}[Partial result of Theorem~\ref{bound}]
\quad
$$\dim_{\mathbb{F}_3} S^{(\dual{\phi})}(\dual{\gp_H}/F) \leq
 \dim_{\mathbb{F}_3} {_\mathrm{N}\mathfrak{U}_{\gext,S'}} \oplus \mu_3(F)
+\dim_{\mathbb{F}_3} {_\mathrm{N}\iclsgp_{\gext,S'}^G[3]}-\dim_{\mathbb{F}_3} \text{\footnotesize
 $\!\frac{\Norm{\gext/F}(\gext_{S'}) \cap \ringint{F,S'}^{*}}{\Norm{\gext/F}(\ringint{\gext,S'}^{*})}$}$$
where ${_\mathrm{N}\mathfrak{U}_{\gext,S'}}$, ${_\mathrm{N}\iclsgp_{\gext,S'}[3]}$, $\gext_{S'}$
 are subgroups of $\ringint{\gext,S'}^{*}/\ringint{\gext,S'}^{*3}$, $\iclsgp_{\gext,S'}[3]$, $\gext^{*}$
 respectively, which are defined in \S\ref{sec_SI}.
The set $S'$ is a subset of all finite primes $\primes_F$ of $F$, defined by
$$S':=\lt\{\prim{p} \in \primes_F \;\Big|\;
 \text{$\prim{p}$ splits completely in $\gext$ and \mbox{\parbox{35ex}{\small
 $3\nu_{\prim{p}}\bigl(a\delta(\gp)\bigr)<\nu_{\prim{p}}\bigl(\Norm{\gext/F}(\gamma(\gp_H))\bigr)$\\
 \quad or \; $\nu_{\prim{p}}\bigl(\Norm{\gext/F}(\gamma(\gp_H))\bigr) \not\equiv 0\!\!\!\!\pmod{3}$}}}\rt\}.$$
Furthermore if $\iclsgp_{\gext,S'}$ is $G$-invariant then
 $\gp_H(F)/\dual{\phi}\bigl(\dual{\gp_H}(F)\bigr) \hookrightarrow
 {_\mathrm{N}\mathfrak{U}_{\gext,S'}} \oplus \mu_3(F)$.
\end{theorem}

As a special case, we can state more precise result.

\begin{theorem}\label{equality}
\quad
For a cyclic cubic field $\gext/\rat$, let $\{1,\,\alpha,\,\alpha^2\}$ be a $\rat$-basis of $\gext$
 that $\gp$ is defined by the polynomial \eqref{defpoly} with.
Here $\alpha \in \gext$ denotes a root of the generic polynomial $\mathfrak{g}(x;t)$
 in \S\ref{genpoly} for some $t \in \rat$.
Let $H$ be the hyperplane $z=0$.
Assume $t \in \rint$, $\pm\Norm{\gext/F}(\alpha-\alpha^{\sigma})=t^2+3t+9 \notin \rat_3^{*3}$.
Then $S'=\emptyset$ and
$$\dim_{\mathbb{F}_3} S^{(\phi)}(\gp_H/\rat)=\dim_{\mathbb{F}_3} S^{(\dual{\phi})}(\dual{\gp_H}/\rat)
=1+\dim_{\mathbb{F}_3} {_\mathrm{N}\iclsgp_{\gext}^G[3]},$$
$$1 \leq \rank_{\rint} \gp_H(\rat) \leq 1+2\dim_{\mathbb{F}_3} {_\mathrm{N}\iclsgp_{\gext}^G[3]}.$$
Furthermore, we also assume that $\iclsgp_{\gext}$ is $G$-invariant. Then
\begin{align*}
\gp_H(\rat)/\dual{\phi}\bigl(\dual{\gp_H}(\rat)\bigr)
 &\simeq {_\mathrm{N}\mathfrak{U}_{\gext}} \simeq \rint/3\rint,\\
\text{\cyr X}(\dual{\gp_H}/\rat)[\dual{\phi}] &\simeq {_\mathrm{N}\iclsgp_{\gext}[3]}.
\end{align*}
Here $\text{\cyr X}$ stands for the Shafarevich-Tate group.
\end{theorem}

The present paper is organized as follows.
First, we will give basic properties of the cubic surface $\gp/F$
 and the family of its hyperplane sections in \S\ref{sec_def}.
Using these properties and a generic polynomial for cyclic cubic extensions (in \S\ref{genpoly}),
 we obtain a Weierstrass form of the cubic curve $\gp_H$ in \S\ref{sec_WE}.
We will discuss, in \S\ref{sec_descent}, descent theory for $\gp_H$ via $3$-isogeny,
 and give some relation between Selmer groups and unit/class groups in \S\ref{sec_SI}.
We also refer corollaries and remarks of the above theorems to the corresponding sections.

\section{First properties of $\gp$ and its hyperplane sections}\label{sec_def}

Fix a cyclic cubic extension $\gext$ over a number field $F$ with Galois group
 $G=\langle \sigma \rangle \simeq \rint/3\rint$, once and for all.
For any $F$-basis $\{1,\,\alpha_1,\,\alpha_2\}$ of $\gext$,
 let $\gp/F$ be a cubic surface in $\mathbb{P}^3$ defined by the polynomial \eqref{defpoly}.
It is easy to verify that the cubic surface $\gp/F$ has three singular points on the hyperplane $w=0$,
 and is birationally equivalent over $\gext$ to the projective plane $\mathbb{P}^2$.

Let $\mathcal{L}(\gp)$ be a set of all those hyperplanes $\{H\}$ defined over $F$ in $\mathbb{P}^3$
 which satisfy $\gp_H(F) \ne \emptyset$. From the definition, we observe the following.

\begin{lemma}\label{basis}
\quad
The family $\mathcal{L}(\gp)$ is invariant under the change of $F$-basis of $\gext$.
\end{lemma}

\vsp

\begin{proof}
\quad
Since any basis transformation for a vector space is linear and defined over a ground field,
 the lemma follows immediately.
\end{proof}

From now on, we choose any $F$-basis of $\gext$ and fix it without loss of generality.
From the definition of $\gp$, the affine piece $\gp_{w \ne 0}:=\gp \cap \{w \ne 0\}$ of the surface $\gp$
 has a structure of linear algebraic groups by a natural correspondence between the group of $F$-rational points
 $\gp_{w \ne 0}(F)=\gp(F)$ and the kernel of the norm map:
\begin{align*}
\gp_{w \ne 0}(F) \quad &\stackrel{\sim}{\longrightarrow} \quad
 \Ker[\gext^* \stackrel{\Norm{\gext/F}}{\longrightarrow} F^*]\\
[x,\,y,\,z,\,1] \quad &\longmapsto \quad x+y\alpha_1+z\alpha_2.
\end{align*}
The rational function $x+y\alpha_1+z\alpha_2$ on $\gp_{w \ne 0}$ also induces an isomorphism
$$\gp_{w \ne 0}(\ringint{F}) \quad \stackrel{\sim}{\longrightarrow} \quad
 \Ker[\ringint{\gext}^* \stackrel{\Norm{\gext/F}}{\longrightarrow} \ringint{F}^*]$$
if the set $\{1,\,\alpha_1,\,\alpha_2\}$ forms an $\ringint{F}$-basis of $\ringint{\gext}$.
For instance this holds for the case that the class number of the ground field $F$ is equal to $1$.
Particularly in the case $F=\rat$, we have $\gp(\rint) \simeq \ringint{\gext}^*/\{\pm 1\}$
 for an integral basis of $\gext$.
Thus, in this case, we can regard the group of units $\ringint{\gext}^{*}$
 as the group of $\rint$-valued points on the variety $\gp$.
This identification is one of the reason why we adopt the algebraic variety $\gp$ in view of the analogy
 between Mordell-Weil groups of elliptic curves and groups of units.
Accordingly it might be of some interest to deduce the arithmetic of the elliptic curves $\{\gp_H\}_H$
 from the cyclic cubic extension $\gext/F$, and that is exactly what we propose to carry out in the present paper.
Note that the above isomorphism can be naturally generalized to the group $\gp(\ringint{F,S})$
 where $\ringint{F,S}$ is a ring of $S$-integers of $F$.
We will see a considerable aspect of the function $x+y\alpha_1+z\alpha_2$ in \S\ref{sec_const},
 which is the most essential part of the paper to obtain an arithmetic relation between the section $\gp_H/F$
 and the ambient variety $\gp/F$ reflecting the cyclic cubic extension $\gext/F$.

For any point $P \in \gp_{w \ne 0}(F)$,
 the left translation $\gp_{w \ne 0}(F) \to \gp_{w \ne 0}(F)\,;\,Q \mapsto P \cdot Q$
 causes the faithful regular representation $\rho$ of degree $3$ and the canonical enlargement $\tilde{\rho}$
\begin{align*}
\tilde{\rho}:\gp_{w \ne 0}(F) \quad \stackrel{\rho}{\longrightarrow} \quad \GL_3(F)
 \quad &\longrightarrow \quad \PGL_4(F)\\
\rho(P) \quad &\longmapsto \quad
\begin{pmatrix}
\rho(P) & 0\\
0 & 1
\end{pmatrix}.
\end{align*}
The representation $\tilde{\rho}$ induces a linear action of the group $\gp_{w \ne 0}(F)$
 on the set $\gp(\closure{F})$
\begin{align*}
\gp_{w \ne 0}(F) \times \gp(\closure{F}) \quad &\longrightarrow \quad \gp(\closure{F})\\
(P,\,[x,\,y,\,z,\,w]) \quad &\longmapsto \quad [x,\,y,\,z,\,w]\cdot\tilde{\rho}(P).
\end{align*}
Note that $\gp(\closure{F})$ is only a set, not a group.
Thus each element in $\gp_{w \ne 0}(F)$ gives an isomorphism on $\gp$,
 and hence we obtain an action of $\gp_{w \ne 0}(F)$ on the family $\mathcal{L}(\gp)$
 defined by $P \cdot \gp_H:=\gp_{\tilde{\rho}(P)(H)}$.
This action makes matter somewhat simple as follows.

\begin{proposition}\label{normalization}
\quad
For arbitrary hyperplane $H \in \mathcal{L}(\gp)$,
 there is a hyperplane $H_0 \in \mathcal{L}(\gp)$ on which the point $\ellunit$ lies
 so as to be $\gp_H \simeq \gp_{H_0}$ over $F$.
\end{proposition}

\vsp

\begin{proof}
\quad
This is almost clear.
For any $H \in \mathcal{L}(\gp)$ with a $F$-rational point $P \in \gp_H(F) \subset \gp_{w \ne 0}(F)$,
 the hyperplane $H_0:=\tilde{\rho}(P^{-1})(H)$ satisfies the property of the statement.
\end{proof}

By Proposition~\ref{normalization},
 the study of the cubic curves $\{\gp_H\}_H$ where $H/F$ varies through $\mathcal{L}(\gp)$
 is reduced to the family of curves $\gp_H$ which possess the point $\ellunit \in H$.
We shall denote by $\mathcal{L}(\ellunit)$ the set of all hyperplanes in $\mathbb{P}^3$ defined over $F$
 on which the point $\ellunit$ lies.
The following lemma gives a criterion for $\gp_H$ being an elliptic curve.

\begin{lemma}\label{singularity}
\quad
The following conditions are equivalent for any $H/F \in \mathcal{L}(\gp)$.
\begin{list}{$\diamond$}{}
\item The cubic curve $\gp_H/F$ is singular.
\item The hyperplane $H/F$ is tangent to the surface $\gp/F$.
\end{list}
Under these conditions, the cubic curve $\gp_H/F$ is birational over $F$ to the projective line.
\end{lemma}

\vsp

\begin{proof}
\quad
Since a linear transformation sends a tangent plane to a tangent plane,
 it suffices to prove the lemma for any hyperplane $H \in \mathcal{L}(\ellunit)$.
We first show that the points on $\gp_H \cap \{w=0\}$ are non-singular.
Let $\iota(x,\,y,\,z,\,w)$ be the rational function $x+y\alpha_1+z\alpha_2 \in \gext[\gp_H]$.
The defining equation of $\gp_H$ is written by
$$f_H(x,\,y,\,z,\,w):=\iota(x,\,y,\,z,\,w)\,\iota^{\sigma}\!(x,\,y,\,z,\,w)\,\iota^{\sigma^2}\!\!(x,\,y,\,z,\,w)-w^3
 \mod a(x-w)+by+cz.$$
It is easily verified that the unique zero of the function $\iota$ on $\gp_H$ is
\begin{align}\label{new}
\ellunit':=[c\alpha_1-b\alpha_2,\,a\alpha_2-c,\,-a\alpha_1+b,\,0]
\end{align}
and hence the set $\gp_H \cap \{w=0\}$ consists of three points
 $\bigl\{\ellunit',\,\ellunit'^{\sigma},\,\ellunit'^{\sigma^2}\bigr\}$.
Thus we have
$$\frac{\partial f_H}{\partial x}(\ellunit')
=\lt(1+\frac{\partial y}{\partial x}(\ellunit')\,\alpha_1+\frac{\partial z}{\partial x}(\ellunit')\,\alpha_2\rt)
\iota^{\sigma}\!(\ellunit')\iota^{\sigma^2}\!\!(\ellunit').$$
Here the partial derivatives are given by
\begin{align*}
\frac{\partial y}{\partial x}(\ellunit')=
\begin{cases}
-\frac{a}{b}	&\text{if $b \ne 0$}\\
0		&\text{if $b=0$},
\end{cases}
\quad\quad
\frac{\partial z}{\partial x}(\ellunit')=
\begin{cases}
-\frac{a}{c}	&\text{if $c \ne 0$}\\
0		&\text{if $c=0$}.
\end{cases}
\end{align*}
Using this yields the non-singularity of the point $\ellunit'$ on $\gp_H$.
Acting the Galois group $\Gal(\gext/F)=\langle \sigma \rangle$ yields
 non-singularity of the points $\ellunit'^{\sigma}$ and $\ellunit'^{\sigma^2}$ at once.

Now we consider the affine piece $\gp_H \cap \{w \ne 0\}$.
Since the Jacobian matrix for $\gp_H$ at a point $P$ may be written by
$$
J_{P}=
\begin{pmatrix}
\frac{\partial f}{\partial x}(P)	&\frac{\partial f}{\partial y}(P)	&\frac{\partial f}{\partial z}(P)\\
a					&b					&c
\end{pmatrix},
$$
the point $P \in \gp_H$ is singular if and only if there is some constant
 $\lambda \in F(P)^{*}$ satisfying
$$(a,\,b,\,c)=\lambda\lt(\frac{\partial f}{\partial x}(P),\,\frac{\partial f}{\partial y}(P),\,
\frac{\partial f}{\partial x}(P)\rt).$$
This implies the first equivalence.
\end{proof}

Recall we have fixed the point $\ellunit$ as a base point of the cubic curve $\gp_H$.
In the proof of Lemma~\ref{singularity}, we observe
$$\gp_H \cap \{w=0\}=\bigl\{\ellunit',\,\ellunit'^{\sigma},\,\ellunit'^{\sigma^2}\bigr\}
 \text{ where $\ellunit'$ is the point \eqref{new}}.$$
We end this section with the following fact.

\begin{lemma}\label{F-point}
\quad
Assume $H \in \mathcal{L}(\ellunit)$ is not tangent to $\gp$. Then $3\ellunit' \in \gp_H(F)$.
Moreover $\ellunit'^{\sigma}-\ellunit'=-(\ellunit'^{\sigma^2}-\ellunit') \in \gp_H(F)[3] \setminus \{\ellunit\}$.
\end{lemma}

\vsp

\begin{proof}
\quad
This follows from direct calculation by the ``chord and tangent processes".
It can be verified that the tangent line at $\ellunit'$ meets $\gp_H$ again at $\ellunit'$,
 namely the intersection is an inflection point.
This implies the first statement of the lemma.
As for the latter, the points $\ellunit',\,\ellunit'^{\sigma},\,\ellunit'^{\sigma^2}$ are collinear
 and the tangent line at $\ellunit$ meets $\gp_H$ again at $3\ellunit'$ by the former proof.
It thus turns out that $\ellunit'+\ellunit'^{\sigma}+\ellunit^{\sigma^2}=3\ellunit'$,
 which yields the required result.
\end{proof}

\section{Weierstrass form of the cubic curve $\gp_H$ over $F$}\label{sec_WE}

Proposition~\ref{normalization} in the previous section allows us to calculate
 a Weierstrass form of the cubic curve $\gp_H$ systematically.
As the $F$-rational point $\ellunit$ on $\gp_H$ is given explicitly,
 the defining equation of $\gp_H/F$ can be transformed into Weierstrass form.
For instance, see Cassels's book \cite{LEC} as for this method.
The proof of Theorem~\ref{WE} sits in the end of this section.
The corollaries below are immediate consequences read off from the Weierstrass form \eqref{we}.

\begin{corollary}[of Theorem~\ref{WE}]\label{basic_inv}
\quad
The discriminant of the Weierstrass form \eqref{we} is
\begin{align*}
\Delta(\gp_H)&=-\Norm{\gext/F}\bigl(\gamma(\gp_H)\bigr)^3\lt(a^3\delta(\gp)^3+3^3\Norm{\gext/F}(\gamma(\gp_H))\rt).\\
\intertext{Assume $H \in \mathcal{L}(\ellunit)$ is not tangent to $\gp$. Then the $j$-invariant is}
j(\gp_H)&=-\frac{a^3\delta(\gp)^3\lt(a^3\delta(\gp)^3+24\,\Norm{\gext/F}(\gamma(\gp_H))\rt)^3}
{\Norm{\gext/F}\bigl(\gamma(\gp_H)\bigr)^3\lt(a^3\delta(\gp)^3+3^3\Norm{\gext/F}(\gamma(\gp_H))\rt)}.
\end{align*}
Especially the family $\{\gp_H/F\}_{H \in \mathcal{L}(\ellunit)}$ contains
 infinitely many isomorphism class of elliptic curves.
\end{corollary}

\vsp

\begin{proof}
\quad
A direct calculation. For example, see \cite{AEC}.
\end{proof}

\begin{corollary}[of Theorem~\ref{WE}]\label{tors}
\quad
Assume $H \in \mathcal{L}(\ellunit)$ is not tangent to $\gp$.
Then $(0,\,0)=-\bigl(0,\,\Norm{\gext/F}(\gamma(\gp_H))\bigr) \in \mathcal{W}(\gp_H)(F)[3]
 \setminus \bigl\{[0,\,1,\,0]\bigr\}$.
\end{corollary}

\vsp

\begin{proof}
\quad
Since these points are inflection points, the statement follows.
\end{proof}

\begin{remark}
\quad
By Corollary~\ref{tors}, there is a dual elliptic curve $\dual{\gp_H}$ defined over $F$ which is isomorphic to
 $\mathcal{W}(\gp_H)/\langle (0,\,0) \rangle$ given by the Weierstrass form $($For example, see {\rm\cite{RK}}$)$
$$y^2+a\,\delta(\gp)xy+3^2\Norm{\gext/F}\bigl(\gamma(\gp_H)\bigr)y
=x^3+\Norm{\gext/F}\bigl(\gamma(\gp_H)\bigr)\lt(a^3\delta(\gp)^3-3^3\Norm{\gext/F}(\gamma(\gp_H))\rt).$$
\end{remark}

\begin{remark}
\quad
The family $\{\gp_H/F\}_{H \in \mathcal{L}(\ellunit)}$ is independent of the choice of $F$-basis of $\gext$,
 but each curve $\gp_H/F$ depends on.
\end{remark}

\begin{remark}
\quad
For arbitrary $(a,\,b,\,c) \in \mathbb{A}^3(F)$,
 define $\gamma(a,\,b,\,c)$ to be the discriminant of the same matrix as $\gamma(\gp_H)$ in Theorem~\ref{WE}.
Then, as a matter of fact, the quantity $\Norm{\gext/F}\bigl(\gamma(a,\,b,\,c)\bigr)$ is
 $\gp_{w \ne 0}(F)$-invariant, namely $\Norm{\gext/F}\bigl(\gamma((a,\,b,\,c)\rho(P))\bigr)
=\Norm{\gext/F}\bigl(\gamma(a,\,b,\,c)\bigr)$ for arbitrary $P \in \gp_{w \ne 0}(F)$.
Especially we have
 $\Norm{\gext/F}\bigl(\gamma(\gp_H)\bigr)=\Norm{\gext/F}\bigl(\gamma(\gp_{\tilde{\rho}(P)(H)})\bigr)$
 for any $P \in \gp_{w \ne 0}(F)$ by taking suitable coefficients of hyperplanes as follows.
Let $[a,\,b,\,c] \in \mathbb{P}^2(\ringint{F})$
 be coefficients of any hyperplane $H/F \in \mathcal{L}(\ellunit)$.
Take $(a',\,b',\,c'):=(a,\,b,\,c)\rho(P)$ for arbitrary $P \in \gp_{w \ne 0}(F)$.
Then the hyperplane $H':=\tilde{\rho}(P)(H)$ has the defining equation $a'(x-x_0w)+b'(y-y_0w)+c'(z-z_0w)=0$
 where $(x_0,\,y_0,\,z_0):=(1,\,0,\,0)^t\!\rho(P)^{-1}$.
Further if $P$ lies on $\gp_H(F)$ then one can observe that $a=a'$.
This fact implies that for any $P \in \gp_H(F)$ the section $\gp_{H'}$ by the hyperplane
 $H':=\tilde{\rho}(P^{-1})(H) \in \mathcal{L}(\ellunit)$ has the Weierstrass form \eqref{we}.
In connection with Proposition~\ref{normalization}, it thus turns out that
 the Weierstrass form \eqref{we} of $\gp_H$ is invariant for the choice of a point on $\gp_H(F)$.
\end{remark}

From now on we shall show necessary lemmas for the proof of Theorem~\ref{WE}.
First, we choose a generic polynomial for $C_3 \simeq \rint/3\rint$-extension over $F$,
$$\bigl(\mathfrak{g}(x)=\bigr)\;\mathfrak{g}(x;t)=x^3+t\,x^2-(t+3)x+1
 \quad \bigl(t \in F,\,\delta(\mathfrak{g})=t^2+3t+9\bigr).$$
We recall basic properties of this polynomial in \S\ref{genpoly} briefly.
By the property, for any cyclic cubic extension $\gext/F$, there is a constant $t \in F$
 such that the roots of $\mathfrak{g}(x;t)$ generates $\gext$ over $F$.
Fix $t \in F$ for $\gext/F$ and $\alpha \in \gext$ as one of the roots.
Let $\gp^0/F$ be a cubic hypersurface in $\mathbb{P}^3$ associated to $\gext/F$ with the $F$-basis
 $\{1,\,\alpha,\,\frac{1}{1-\alpha}\}$ of $\gext$ defined by the same polynomial as \eqref{defpoly}.

For this $\gp^0/F$ and any hyperplane $H \in \mathcal{L}(\ellunit)$,
 we prove as follows that the curve $\gp_H^0$ is birational over $F$ to
 either the cubic curve $\mathcal{W}(\gp_H^0)$ defined by the same Weierstrass form as \eqref{we}
 or $\mathbb{P}^1$ by dividing it into $3$-cases according to embeddings of $\gp_H^0$ into $\mathbb{P}^2$
 as plane cubic curves.

\begin{lemma}\label{embed1}
\quad
Let $H:a(x-w)+by+cz=0$ be a hyperplane in $\mathbb{P}^3$ defined over $F$.
Assume $at+3c \ne 0$.
Then there is a birational map given by
\begin{align*}
\theta:\gp_H^0 \quad &\stackrel{\sim}{\longrightarrow} \quad \mathcal{W}(\gp_H^0);\\
[x,\,y,\,z,\,w] \quad &\longmapsto \quad [\theta_X,\,\theta_Y,\,\theta_Z],\\
[\theta^{(-1)}_X,\,\theta^{(-1)}_Y,\,\theta^{(-1)}_Z,\,\theta^{(-1)}_W]
 \quad &\text{\,\reflectbox{$\longmapsto$}\:\,} \quad [x,\,y,\,z],
\end{align*}
$$
\begin{cases}
\theta_X
=u_0\,u_1^2\lt(\frac{y-3\,c_0\,w}{h_0(x,y,z,w)w}+\frac{3A^2\delta(\mathfrak{g})+c_7}{c_1}\rt)\!w,\\
\theta_Y
=u_0\,u_1^3\lt(\frac{2}{3}\,g_3\bigl[\frac{y-3\,c_0\,w}{h_0(x,y,z,w)w}\bigr]\,h_0(x,y,z,w)
+g_2\bigl[\frac{y-3\,c_0\,w}{h_0(x,y,z,w)w}\bigr]\rt)w
-\frac{a\,\delta(\mathfrak{g})\theta_X-\Norm{\gext/F}(\gamma(\gp_H^0))w}{2},\\
\theta_Z=w,
\end{cases}
$$
$$
\begin{cases}
\theta^{(-1)}_X=\bigl(1+t(-A-B\,h_1(x,\,z)+h_1(x,\,z))\bigr)\,h_2(x,\,y,\,z)z+\bigl(1+t\,c_0(1-B)\bigr)z,\\
\theta^{(-1)}_Y=3\bigl(h_1(x,\,z)\,h_2(x,\,y,\,z)+c_0\bigr)z,\\
\theta^{(-1)}_Z=3\bigl(-(A+B\,h_1(x,\,z))\,h_2(x,\,y,\,z)-B\,c_0\bigr)z,\\
\theta^{(-1)}_W=z.
\end{cases}
$$
Here the functions $h_0(x,\,y,\,z,\,w),\,h_1(x,\,z),\,h_2(x,\,y,\,z) \in F(x,\,y,\,z,\,w)$ are
$$h_0(x,\,y,\,z,\,w)=3\bigl(\frac{x}{w}-1\bigr)-t\frac{y}{w}-t\frac{z}{w}, \quad h_1(x,\,z)=\frac{x}{u_0\,u_1^2\,z}
-\frac{3A^2\delta(\mathfrak{g})+c_7}{c_1},$$
$$h_2(x,\,y,\,z)=
\frac{\bigl(y+2^{-1}(a\,\delta(\mathfrak{g})x-\Norm{\gext/F}(\gamma(\gp_H^0))z)\bigr)-g_2[h_1(x,\,z)]u_0\,u_1^3z}
{2\,g_3[h_1(x,\,z)]u_0\,u_1^3z},$$
and the functions $g_2[X],\,g_3[X]$ are elements in the ring of polynomials $F[X]$ defined by
\begin{align*}
g_3[X]&=c_1\delta(\mathfrak{g})X^3+3\,c_2\,\delta(\mathfrak{g})X^2+3A\,c_3\,\delta(\mathfrak{g})X+c_4,\\
g_2[X]&=6(B^2+B+1)\delta(\mathfrak{g})X^2+3\frac{c_5}{c_1}\delta(\mathfrak{g})X
+3A\frac{c_6}{c_1}\delta(\mathfrak{g})+3,
\end{align*}
where $(A,\,B)=(\frac{a}{at+3c},\,\frac{at+3b}{at+3c})$ and the other constants are as follows.
$$c_0=\frac{3(B^2+B+1)}{c_1}, \quad u_0=-12\delta(\mathfrak{g})c_1, \quad u_1=\frac{at+3c}{18},$$
\begin{align*}
c_1&=(B+2)^3\mathfrak{g}\!\lt(\frac{-B+1}{B+2}\rt),\\
c_2&=A\lt(B^2(2\,t+3)+2B(t+6)-t+3\rt)-(B^2+B+1),\\
c_3&=A\bigl(B(2\,t+3)+t+6\bigr)-2B-1,\\
c_4&=(3A)^3\mathfrak{g}\!\lt(\frac{1-At}{3A}\rt),\\
c_5&=\text{\footnotesize $A\lt(4B^4(2\,t+3)+B^3(16\,t+51)+3B^2(7\,t+24)+B(13\,t+87)-4\,t+21\rt)-6(B^2+B+1)^2$},\\
c_6&=A\lt(2B^3(2\,t+3)+3B^2(2\,t+3)+6B(2\,t+3)+5\,t+21\rt)-3\lt((B+1)^3+B^3\rt),\\
c_7&=A\lt(B^2(2\,t+3)+2B(t+6)-t+3\rt)-2(B^2+B+1).
\end{align*}
\end{lemma}

\vsp

\begin{proof}
\quad
On account of technical reason, we first make the substitution
\begin{align*}
[x,\,y,\,z,\,w] \quad \stackrel{i}{\longmapsto} \quad [x,\,y,\,z,\,w] \cdot
\begin{pmatrix}
3  & 0 & 0 & 0 \\
-t & 1 & 0 & 0 \\
-t & 0 & 1 & 0 \\
0  & 0 & 0 & 3
\end{pmatrix}
\end{align*}
which maps the point $\ellunit$ to $\ellunit$.
The substitution $i$ maps the hyperplane $H:a(x-w)+by+cz=0$ to
 the hyperplane $i(H):a(x-w)+(at+3b)y+(at+3c)z=0$.
We denote the polynomial $f \circ i^{-1}$ defining $i(\gp^0)$ by ${^i\!f}$
 where $f$ denotes the polynomial defining $\gp^0$.
From the assumption $at+3c \ne 0$, we have the cubic form ${^i\!f_H}(x,\,y,\,w):={^i\!f}(x,\,y,\,-A(x-w)-By,\,w)$
which defines the plane cubic curve $i(\gp_H^0)=i(\gp^0) \cap i(H) \hookrightarrow \mathbb{P}^2$
 isomorphic to $\gp_H^0$.
We may regard a point $[x,\,y,\,w]$ on $\mathbb{P}^2$ as an affine point $(x/w,\,y/w) \in \mathbb{A}^2$ if $w \ne 0$.
Then $\ellunit$ may be identified with $(1,\,0)$.

The tangent at $(1,\,0) \in i(\gp_H^0)$ meets $i(\gp_H^0)$ again at a $F$-rational point $P$, say.
It is easily calculated that $P=(1,\,c_0) \in i(\gp_H^0)$. Let $g_H(x,\,y):={^i\!f_H}(x+1,\,y+c_0,\,1)$.
Using the method in Cassels's book \cite{LEC}, there is a birational map $\Theta$ over $F$ from
 the cubic curve $i(\gp_H^0)$ to a curve $\mathcal{C}$ whose defining equation of the affine piece $\{w \ne 0\}$
 is given by
$$y^2=g_2[x]^2-4\,g_1[x]\,g_3[x]$$
where $g_d[x]$ denotes the homogeneous part of degree $d$ of the polynomial $g_H(1,\,x)$.
The map $\Theta$ is given by
\begin{align*}
\Theta:i(\gp_H^0) \quad &\stackrel{\sim}{\longmapsto} \quad \mathcal{C}\\
(x,\,y) \quad &\longmapsto \quad \text{\small $\lt(\frac{y-c_0}{x-1},\,2\,g_3\!\!\lt[\frac{y-c_0}{x-1}\rt](x-1)
+g_2\!\!\lt[\frac{y-c_0}{x-1}\rt]\rt)$},\\
\text{\small $\lt(\frac{y-g_2[x]}{2\,g_3[x]}+1,\,\frac{y-g_2[x]}{2\,g_3[x]}x+c_0\rt)$}
 \quad &\text{\,\reflectbox{$\longmapsto$}\:\,} \quad (x,\,y).
\end{align*}
Having done it, by the substitution
$$(x,\,y) \quad \stackrel{l}{\longmapsto} \quad
(l_x,\,l_y):=\text{\footnotesize $\lt(u_0\,u_1^2\lt(x+\frac{3A^2\delta(\mathfrak{g})+c_7}{c_1}\rt),\,
u_0\,u_1^3\,y-\frac{a\,\delta(\mathfrak{g})l_x-\Norm{\gext/F}\bigl(\gamma(\gp_H^0)\bigr)}{2}\rt)$},$$
we achieve the Weierstrass form defining $\mathcal{W}(\gp_H^0)$.
Thus the composed map $\theta:=l \circ \Theta \circ i;\,\gp_H^0 \stackrel{\sim}{\to} \mathcal{W}(\gp_H^0)$
 is the desired one.
Since the all is standard manner, the precise proof is omitted.
\end{proof}

\begin{lemma}\label{embed2}
\quad
Let $H:a(x-w)+by+cz=0$ be a hyperplane in $\mathbb{P}^3$ defined over $F$.
Assume $at+3c=0$ and $at+3b \ne 0$.
Then there is a birational map given by
\begin{align*}
\theta:\gp_H^0 \quad &\stackrel{\sim}{\longrightarrow} \quad \mathcal{W}(\gp_H^0);\\
[x,\,y,\,z,\,w] \quad &\longmapsto \quad [\theta_X,\,\theta_Y,\,\theta_Z],\\
[\theta^{(-1)}_X,\,\theta^{(-1)}_Y,\,\theta^{(-1)}_Z,\,\theta^{(-1)}_W]
 \quad &\text{\,\reflectbox{$\longmapsto$}\:\,} \quad [x,\,y,\,z],
\end{align*}
$$
\begin{cases}
\theta_X
=u_0\,u_1^2\lt(\frac{z-3\,c_0\,w}{h_0(x,y,z,w)w}+\frac{3A^2\delta(\mathfrak{g})+c_7}{c_1}\rt)\!w,\\
\theta_Y
=u_0\,u_1^3\lt(\frac{2}{3}\,g_3\bigl[\frac{z-3\,c_0\,w}{h_0(x,y,z,w)w}\bigr]\,h_0(x,y,z,w)
+g_2\bigl[\frac{z-3\,c_0\,w}{h_0(x,y,z,w)w}\bigr]\rt)w
-\frac{a\,\delta(\mathfrak{g})\theta_X-\Norm{\gext/F}(\gamma(\gp_H^0))w}{2},\\
\theta_Z=w,
\end{cases}
$$
$$
\begin{cases}
\theta^{(-1)}_X=\bigl(1+t(-A+h_1(x,\,z))\bigr)\,h_2(x,\,y,\,z)z+(1+t\,c_0)z,\\
\theta^{(-1)}_Y=-3A\,h_2(x,\,y,\,z)z,\\
\theta^{(-1)}_Z=3\bigl(h_1(x,\,z)\,h_2(x,\,y,\,z)+c_0\bigr)z,\\
\theta^{(-1)}_W=z.
\end{cases}
$$
Here the functions $h_0(x,\,y,\,z,\,w),\,h_1(x,\,z),\,h_2(x,\,y,\,z) \in F(x,\,y,\,z,\,w)$
 and $g_3[X] \in F[X]$ have the same notation as in Lemma~\ref{embed1}, but
 $g_2[X]=6\,\delta(\mathfrak{g})X^2+3\frac{c_5}{c_1}\delta(\mathfrak{g})X+3A\frac{c_6}{c_1}\delta(\mathfrak{g})+3$
 and the constants are redefined as follows.
$$A=\frac{a}{at+3b}, \quad c_0=\frac{3}{c_1}, \quad u_0=-12\delta(\mathfrak{g})c_1, \quad u_1=\frac{at+3b}{18},$$
$$c_1=-2\,t-3, \quad c_2=-A(t+6)-1, \quad c_3=A(t-3)-1, \quad c_4=(3A)^3\mathfrak{g}\!\lt(\frac{1-At}{3A}\rt),$$
$$c_5=-A(4\,t+33)-6, \quad c_6=A(5\,t-6)-3, \quad c_7=-A(t+6)-2.$$
\end{lemma}

\vsp

\begin{proof}
\quad
This is almost similar to the previous lemma.
The substitution $i$ defined in the previous lemma maps the hyperplane $H:a(x-w)+by+cz=0$ to
 the hyperplane $i(H):a(x-w)+(at+3b)y+(at+3c)z=0$.
From the assumption $at+3c=0,\,at+3b \ne 0$, we have the cubic form
 ${^i\!f_H}(x,\,z,\,w):={^i\!f}(x,\,-A(x-w),\,z,\,w)$ defining $i(\gp_H^0) \hookrightarrow \mathbb{P}^2$.
Applying the same method yields the birational map
\begin{align*}
\Theta:i(\gp_H^0) \quad &\stackrel{\sim}{\longmapsto} \quad \mathcal{C}\\
(x,\,z) \quad &\longmapsto \quad \text{\small $\lt(\frac{z-c_0}{x-1},\,2\,g_3\!\!\lt[\frac{z-c_0}{x-1}\rt](x-1)
+g_2\!\!\lt[\frac{z-c_0}{x-1}\rt]\rt)$},\\
\text{\small $\lt(\frac{y-g_2[x]}{2\,g_3[x]}+1,\,\frac{y-g_2[x]}{2\,g_3[x]}x+c_0\rt)$}
 \quad &\text{\,\reflectbox{$\longmapsto$}\:\,} \quad (x,\,y).
\end{align*}
where the curve $\mathcal{C}$ is defined by the equation
$$y^2=g_2[x]^2-4\,g_1[x]\,g_3[x].$$
Here $g_d[x]$ denotes the homogeneous part of degree $d$ of the polynomial $g_H(1,\,x)$
 where $g_H(x,\,z):={^i\!f_H}(x+1,\,z+c_0,\,1)$.
In order to obtain the required Weierstrass form, we make the substitution
$$(x,\,y) \quad  \stackrel{l}{\longmapsto} \quad
(l_x,\,l_y):=\text{\footnotesize $\lt(u_0\,u_1^2\lt(x+\frac{3A^2\delta(\mathfrak{g})+c_7}{c_1}\rt),\,
u_0\,u_1^3\,y-\frac{a\,\delta(\mathfrak{g})l_x-\Norm{\gext/F}\bigl(\gamma(\gp_H^0)\bigr)}{2}\rt)$}.$$
This yields the desired birational map
 $\theta:=l \circ \Theta \circ i;\,\gp_H^0 \stackrel{\sim}{\to} \mathcal{W}(\gp_H^0)$.
\end{proof}

\begin{remark}
\quad
As for the precise proof of the above lemmas, it might be required to use the explicit calculation
\begin{align*}
&\Norm{\gext/F}\bigl(\gamma(\gp_H^0)\bigr)\\
&=\text{\footnotesize $\frac{\delta(\mathfrak{g})}{3^3}\Bigl(-(at+3c)^3\mathfrak{g}(-B)-a^3\delta(\mathfrak{g})^2
+a\bigl((at+3c)^2+(at+3c)(at+3b)+(at+3b)^2\bigr)\delta(\mathfrak{g})\Bigr)$}\\
&=\text{\footnotesize $\delta(\mathfrak{g})\Bigl(-c^3\mathfrak{g}\bigl(\frac{-b}{c}\bigr)-a^3(2t+3)+a^2t(bt+3b+3c)
+a\bigl(b^2(2t+3)-bc(t^2+t-3)-c^2(t-3)\bigr)\Bigr)$}.
\end{align*}
\end{remark}

\begin{lemma}\label{embed3}
\quad
Let $H:a(x-w)+by+cz=0$ be a hyperplane in $\mathbb{P}^3$ defined over $F$.
Assume $at+3c=at+3b=0$. Then the hyperplane $H$ is tangent to $\gp^0$ at $\ellunit$
 $($i.e. $\gp_H^0$ is singular from Lemma~\ref{singularity}$)$,
 and there is a birational map given by
\begin{align*}
\theta:\gp_H^0 \quad &\stackrel{\sim}{\longrightarrow} \quad \mathbb{P}^1;\\
[x,\,y,\,z,\,w] \quad &\longmapsto \quad [y,\,z],\\
\lt[-\frac{3t(x^3+y^3)}{(x+y)^3\mathfrak{g}\bigl(\frac{2y-x}{x+y}\bigr)}+1,\,
-\frac{9x(x^2-xy+y^2)}{(x+y)^3\mathfrak{g}\bigl(\frac{2y-x}{x+y}\bigr)},\,
-\frac{9y(x^2-xy+y^2)}{(x+y)^3\mathfrak{g}\bigl(\frac{2y-x}{x+y}\bigr)},\,1\rt]
 \quad &\text{\,\reflectbox{$\longmapsto$}\:\,} \quad [x,\,y].
\end{align*}
\end{lemma}

\vsp

\begin{proof}
\quad
Since $a \ne 0$ by assumption, the defining equation of $H$ is also written as $3(x-w)-ty-tz=0$.
But we have $\bigl[\frac{\partial f}{\partial x}(\ellunit),
\,\frac{\partial f}{\partial y}(\ellunit),\,\frac{\partial f}{\partial z}(\ellunit),
\,\frac{\partial f}{\partial w}(\ellunit)\bigr]=[3,\,-t,\,-t,\,-3]$,
 which implies that $H$ is tangent to $\gp^0$ at $\ellunit$.
It is directly verified the birationality of the map.
Note that the equality $(y+z)^3\mathfrak{g}\bigl(\frac{2z-y}{y+z}\bigr)=-9(y^2-yz+z^2)$
 holds for any $[x,\,y,\,z,\,w] \in \gp_H^0$.
\end{proof}

\begin{remark}
\quad
As for the above, the unique singular point on the cubic curve $\gp_H^0$ is $\ellunit$.
\end{remark}

\begin{remark}\label{tangent}
\quad
From Corollary~\ref{basic_inv}, the cubic curve $\mathcal{W}(\gp_H^0)/F$ is singular
 if and only if $\Delta(\gp_H^0)=0$
 which is equivalent to $A=\frac{\mathfrak{g}(-B)}{(B^2+B+1)\delta(\mathfrak{g})}$ in the case $at+3c \ne 0$.
Thus the equation defining $H$ is described as
$$\frac{\mathfrak{g}(-B)}{(B^2+B+1)\delta(\mathfrak{g})}\bigl(3(x-w)-ty-tz\bigr)+By+z=0.$$
By direct calculation, one can find that the hyperplane $H$ coincides with the tangent plane to $\gp^0$ at the point
$$\text{\footnotesize $^t\!\!
\begin{pmatrix}
\bigl(-B^6(t+2)-3B^5(t+3)-6B^4(t+4)+B^3(t^2+t-23)+3B^2(t-3)-3B-2\bigr)\mathfrak{g}(-B)^{-2}\\
(B^2+B+1)\bigl(-B^4-2B^3-6B^2+2B(t-1)+t+2\bigr)\mathfrak{g}(-B)^{-2}\\
(B^2+B+1)\bigl(B^4(t+1)+2B^3(t+4)+6B^2+2B+1\bigr)\mathfrak{g}(-B)^{-2}
\end{pmatrix}$}.
$$
Therefore this is the only singular point on $\gp_H^0$.
In the case $at+3c=0,\,at+3b \ne 0$, since $\Delta(\gp_H^0)=0$ is equivalent to $A=-\frac{1}{\delta(\mathfrak{g})}$,
 the hyperplane is
$$-\frac{1}{\delta(\mathfrak{g})}(x-w)+y=0.$$
This is the tangent plane to $\gp^0$ at the point $(-t-2,\,-1,\,t+1)$, which is singular on $\gp_H^0$.
These also assure Lemma~\ref{singularity} once again.
\end{remark}

\begin{lemma}\label{identity}
\quad
Assume $H \in \mathcal{L}(\ellunit)$ is not tangent to $\gp^0$.
Then $\theta(\ellunit)=[0,\,1,\,0]$
 where $\theta:\gp_H^0 \to \mathcal{W}(\gp_H^0)$ is the map in Lemma~\ref{embed1} or \ref{embed2}.
Especially $\theta$ is a group isomorphism.
\end{lemma}

\vsp

\begin{proof}
\quad
It suffices to prove the lemma for the map $\Theta$ defined in the proof of Lemma~\ref{embed1} or \ref{embed2}
 since $\theta=l \circ \Theta \circ i$ and $i(\ellunit)=\ellunit,\,l([0,\,1,\,0])=[0,\,1,\,0]$.
First assume $at+3c \ne 0$.
Under the assumption, the ring $F[i(\gp_H^0)]_{\ellunit}$ of regular functions on $i(\gp_H^0)$ at $\ellunit$
 is a discrete valuation ring.
Since the $x$-coordinate of the map $\Theta$ is the function $\frac{y-c_0}{x-1}$ which has a pole at $\ellunit$,
 multiplying an appropriate power of a uniformizer for $F[i(\gp_H^0)]_{\ellunit}$
 to the projective coordinate of $\Theta$ yields $\Theta(\ellunit)=[0,\,1,\,0]$.
In the case $at+3c=0$ and $at+3b \ne 0$, the $x$-coordinate of the map $\Theta$ is the function $\frac{z-c_0}{x-1}$.
We also have $\Theta(\ellunit)=[0,\,1,\,0]$ similarly. The statement is now proven.
\end{proof}

\vsp

\begin{proof}[Proof of Theorem~\ref{WE}]
\quad
Fix arbitrary cyclic cubic extension $\gext/F$.
Let $\gp$ be a cubic hypersurface associated to $\gext/F$ with a $F$-basis $\{1,\alpha_1,\alpha_2\}$ of $\gext$
 defined by the polynomial \eqref{defpoly} and $\gp^0$ as above.
Then the transformation matrix $M \in \GL_3(F)$ with $(1,\,\alpha,\,\alpha^{\sigma})M=(1,\,\alpha_1,\,\alpha_2)$
 defines a regular transformation $\widetilde{M}$ on $\mathbb{P}^3$ given by $[x,\,y,\,z,\,w] \mapsto [x,\,y,\,z,\,w]
\left(\begin{smallmatrix}
^t\!M & 0\\
0 & 1
\end{smallmatrix}\right)$,
whose restriction induces an isomorphism $\gp_H \stackrel{\sim}{\to} \gp_{H'}^0$
where $H'=\widetilde{M}(H)$ has the form $a'(x-w)+b'y+c'z=0$ with $(a',\,b',\,c')M=(a,\,b,\,c)$.
We thus obtain $\delta(\gp^0)\,|M|=\delta(\gp),\,\gamma(\gp_{H'}^0)\,|M|=\gamma(\gp_H)$ and $a'=a$.
From this and Lemma~\ref{embed1}, \ref{embed2}, the composed map
\begin{align}\label{vartheta}
\vartheta:\gp_H \stackrel{\widetilde{M}}{\longrightarrow} \gp_{H'}^0 \stackrel{\theta}{\longrightarrow}
 \mathcal{W}(\gp_{H'}^0) \stackrel{m}{\longrightarrow} \mathcal{W}(\gp_H)
\end{align}
is the required one. Here $m$ denotes the map $(x,\,y) \mapsto \bigl(|M|^2x,\,|M|^3y\bigr)$.
It follows from Lemma~\ref{identity} that $\vartheta(\ellunit)=[0,\,1,\,0]$.
\end{proof}

\begin{proposition}\label{T}
\quad
$\vartheta(\ellunit'^{\sigma}-\ellunit')=\pm(0,\,0)$.
\end{proposition}

\vsp

\begin{proof}
\quad
From \eqref{vartheta},
 it suffices to prove the statement for the map $\theta:\gp_H^0 \to \mathcal{W}(\gp_H^0)$
 since $\widetilde{M}(c\alpha_1-b\alpha_2,\,a\alpha_2-c,\,-a\alpha_1+b,\,0)
=[c'\alpha-b'\frac{1}{1-\alpha},\,a'\frac{1}{1-\alpha}-c',\,-a'\alpha+b',\,0]$.
Let
 $V_F:=\bigl\{\alpha \in \closure{F} \setminus F \;|\; \text{$\mathfrak{g}(\alpha;t)=0$ for some $t \in F$}\bigr\}$,
 and for any $\alpha \in V_F$ let $S_{\alpha} \subset \mathbb{A}^3$ be a hypersurface in variables
 $a,\,b,\,c$ defined by the equation (cf. Remark~\ref{tangent})
$$\lt(-\Norm{\gext/F}\bigl(\gamma(\gp_H^0)\bigr)^{-3}\Delta(\gp_H^0)=\rt)
a^3\delta(\gp^0)^3+3^3\Norm{\gext/F}\bigl(\gamma(\gp_H^0)\bigr)=0$$
where $\gp^0$ is associated to $F(\alpha)/F$ with the $F$-basis $\{1,\,\alpha,\,\frac{1}{1-\alpha}\}$ of $F(\alpha)$
 and $H$ is a hyperplane $a(x-w)+by+cz=0$.
Let $\ellunit'':=[c\frac{1}{1-\alpha}-b(1-\frac{1}{\alpha}),\,a(1-\frac{1}{\alpha})-c,\,-a\frac{1}{1-\alpha}+b,\,0]$
 on $\gp_H^0$ and $P:=\ellunit''-\ellunit'$.
Note that $\ellunit''=\ellunit'^{\sigma}$ or $\ellunit'^{\sigma^2}$
 since $\alpha^{\sigma}=\frac{1}{1-\alpha}$ or $1-\frac{1}{\alpha}$ from \S\ref{genpoly}.
Thus $P \in \gp_H^0(F) \setminus \{\ellunit\}$ by Lemma~\ref{F-point}.
For the map $\theta=[\theta_X,\,\theta_Y,\,\theta_Z]$ as in Lemma~\ref{embed1} or \ref{embed2},
 we can regard $(\theta_x,\,\theta_y):=\bigl(\frac{\theta_X}{\theta_Z}(P),\,\frac{\theta_Y}{\theta_Z}(P)\bigr)$
 as functions in variables $a,\,b,\,c,\,\alpha$,
 namely $\theta_x,\,\theta_y \in \rat(x_1,\,x_2,\,x_3,\,x_4)$.
This makes the rational map
\begin{align*}
\varphi:\mathbb{A}^4 \quad &\longrightarrow \quad \mathbb{A}^2\\
(a,\,b,\,c,\,\alpha) \quad &\longmapsto \quad
 \bigl(\theta_x(a,\,b,\,c,\,\alpha),\,\theta_y(a,\,b,\,c,\,\alpha)\bigl).
\end{align*}
For any $\alpha_0 \in V_{\rat}$,
 let $\varphi|_{\alpha_0}:\mathbb{A}^3 \to \mathbb{A}^2$ be the restriction map to $\alpha=\alpha_0$.
Then the Weil pairing tells us $\varphi|_{\alpha_0}(a,\,b,\,c)=\pm(0,\,0)$
 on $\mathbb{A}^3(\rat) \setminus S_{\alpha_0}(\rat)$
 by using Lemma~\ref{F-point}, Corollary~\ref{tors} and $\zeta_3 \notin \rat$.
Thus the rational functions $\theta_x|_{\alpha_0},\,\theta_y|_{\alpha_0}$ are constant
 for infinitely many values $a,\,b,\,c$ respectively, and hence
 $\theta_x|_{\alpha_0},\,\theta_y|_{\alpha_0}$ must be constants in $\closure{\rat}(x_1,\,x_2,\,x_3)$.
It turns out that the map $\varphi$ depends only on the variable $x_4$.
However $\varphi$ also takes $\pm(0,\,0)$ for infinitely many $\alpha \in V_{\rat}$,
 which implies that $\varphi$ is a constant map taking the image either $(0,\,0)$ or $-(0,\,0)$.
Especially, for any number field $F$,
 the evaluation of $\varphi$ for $\alpha \in V_F$ and $(a,\,b,\,c) \in \mathbb{A}^3(F) \setminus S_{\alpha}(F)$
 yields the required result.
\end{proof}

\section{Descent theory for $\gp_H$ via $3$-isogeny}\label{sec_descent}

We shall now discuss descent theory for the elliptic curve $\gp_H$ over $F$.
Fix $\ellunit$ as a base point of $\gp_H$.
Then the set of $F$-rational points $\gp_H(F)$, often called Mordell-Weil group,
 is a finitely generated abelian group with group unit $\ellunit$,
 whose structure does not depend on the choice of a base point.
It is one of the main interest to know the rank of $\gp_H(F)$.
If one could calculate the quotient group $\gp_H(F)/m\gp_H(F)$ for some integer $m>1$,
 one could find the rank of $\gp_H(F)$.
Descent theory gives an upper bound for the size of the quotient group $\gp_H(F)/m\gp_H(F)$,
 which often gives the exact size.
At the beginning, we make notation used in the present paper
 and recall some basic objects on descent theory in the case $m=3$, which we call $3$-descent.

From Proposition~\ref{T}, we may assume $\vartheta(\ellunit'^{\sigma}-\ellunit')=(0,\,0)$ without loss of generality
 by the proper substitution either
 $\vartheta \leftrightarrow [-1] \circ \vartheta$ or $\sigma \leftrightarrow \sigma^2$.

Let $\phi:\gp_H \to \dual{\gp_H}$ be a $3$-isogeny defined over $F$
 with the dual elliptic curve $\dual{\gp_H}:=\gp_H/\langle T \rangle$ where $T:=\ellunit'^{\sigma}-\ellunit'$,
 and let $\dual{\phi}:\dual{\gp_H} \to \gp_H$ be its dual isogeny which is also defined over $F$.
For the short exact sequence
 $0 \to \gp_H[\phi] \to \gp_H(\closure{F}) \stackrel{\phi}{\to} \dual{\gp_H}(\closure{F}) \to 0$
 as $G_F$-modules where $G_F$ denotes an absolute Galois group of $F$,
 taking Galois cohomology yields the exact sequence
$$0 \longrightarrow \dual{\gp_H}(F)/\phi\bigl(\gp_H(F)\bigr) \longrightarrow H^1(F,\,\gp_H[\phi])
 \longrightarrow H^1(F,\,\gp_H)[\phi] \longrightarrow 0.$$
Considering the above locally again, we also have a similar localized exact sequence.
For the completion $F_{\prim{p}}$ at each prime $\prim{p}$ of $F$,
 since an embedding $\closure{F} \hookrightarrow \closure{F_{\prim{p}}}$, which may be fixed once and for all,
 induces restriction maps of Galois cohomology, we obtain the following commutative diagram.
$$
\begin{CD}
0 @>>> \dual{\gp_H}(F)/\phi\bigl(\gp_H(F)\bigr)	@>{\delta}>>	H^1(F,\,\gp_H[\phi])
  @>>> H^1(F,\,\gp_H)[\phi]			@>>>		0\\
@.     @VV{\res_{\prim{p}}}V					@VV{\res_{\prim{p}}}V
       @VV{\res_{\prim{p}}}V					@.\\
0 @>>> \dual{\gp_H}(F_{\prim{p}})/\phi\bigl(\gp_H(F_{\prim{p}})\bigr)
  @>{\delta_{\prim{p}}}>> H^1(F_{\prim{p}},\,\gp_H[\phi])	@>>> H^1(F_{\prim{p}},\,\gp_H)[\phi]	@>>> 0
\end{CD}
$$
where $\delta,\,\delta_{\prim{p}}$ stand for connecting homomorphisms of Galois cohomology.
The $\phi$-Selmer group of the elliptic curve $\gp_H$ is a finite subgroup of $H^1(F,\,\gp_H[\phi])$ defined by
$$S^{(\phi)}(\gp_H/F)=\text{\small $\lt\{\eqcls{\xi} \in H^1(F,\,\gp_H[\phi]) \;\Big|\;
 \res_{\prim{p}}(\eqcls{\xi}) \in \Image \delta_{\prim{p}}
 \text{ for any finite/infinite prime $\prim{p}$ of $F$}\rt\}$}.$$
It is clear by the above diagram that the group $S^{(\phi)}(\gp_H/F)$ contains
 $\dual{\gp_H}(F)/\phi\bigl(\gp_H(F)\bigr)$ as subgroup.
The gap between these groups is represented by
 the $\phi$-kernel of the Shafarevich-Tate group, which sits in the exact sequence
$$0 \longrightarrow \dual{\gp_H}(F)/\phi\bigl(\gp_H(F)\bigr) \longrightarrow S^{(\phi)}(\gp_H/F)
 \longrightarrow \text{\cyr X}(\gp_H/F)[\phi] \longrightarrow 0.$$
Oppositely, one can define the $\dual{\phi}$-Selmer group $S^{(\dual{\phi})}(\dual{\gp_H}/F)$ and
 the $\dual{\phi}$-kernel of the Shafarevich-Tate group $\text{\cyr X}(\dual{\gp_H}/F)[\dual{\phi}]$
 on interchanging the role of the isogenies $\phi,\,\dual{\phi}$.

A relation between the finite groups $\dual{\gp_H}(F)/\phi\bigl(\gp_H(F)\bigr)$
 and $\gp_H(F)/\dual{\phi}\bigl(\dual{\gp_H}(F)\bigr)$ is described by the exact sequence
$$0 \longrightarrow \frac{\dual{\gp_H}(F)[\dual{\phi}]}{\phi\bigl(\gp_H(F)[3]\bigr)} \longrightarrow
 \frac{\dual{\gp_H}(F)}{\phi\bigl(\gp_H(F)\bigr)} \stackrel{\dual{\phi}}{\longrightarrow}
 \frac{\gp_H(F)}{3\gp_H(F)} \longrightarrow \frac{\gp_H(F)}{\dual{\phi}\bigl(\dual{\gp_H}(F)\bigr)}
 \longrightarrow 0.$$
It is easily seen that $\frac{\dual{\gp_H}(F)[\dual{\phi}]}{\phi(\gp_H(F)[3])} \simeq \rint/3\rint$
 if and only if $\zeta_3 \in F$ and $\gp_H(F)[3]=\gp_H(F)[\phi] \simeq \rint/3\rint$.
From this we have the formula
\begin{align}\label{rank}
\rank_{\rint} \gp_H(F)=\dim_{\mathbb{F}_3} \frac{\gp_H(F)}{\dual{\phi}\bigl(\dual{\gp_H}(F)\bigr)}
+\dim_{\mathbb{F}_3} \frac{\dual{\gp_H}(F)}{\phi\bigl(\gp_H(F)\bigr)}-\dim_{\mathbb{F}_3} \mu_3(F)-1.
\end{align}
Thus the Selmer groups give an upper bound for the rank of the Mordell-Weil group.
From now on, to estimate the size of the rank, we consider the Selmer groups.

\subsection{On the Selmer group $S^{(\phi)}(\gp_H/F)$}

Since $\gp_H[\phi] \subset \gp_H(F)$, one sees that
$$H^1(F,\,\gp_H [\phi]) \simeq H^1(F,\,C_3)=\Hom\bigl(G_F,\,C_3\bigr).$$
Here the Galois action on the cyclic group $C_3 \simeq \rint/3\rint$ is regarded as trivial one.
Define the set of all cyclic cubic extensions over $F$ together with $F$ by
$$\mathfrak{Gal}(C_3/F):=
\lt\{L/F\text{: Galois extension} \;\Big|\; \Gal(L/F) \hookrightarrow C_3\rt\}\big/
\bigl\{\text{isomorphisms}\bigr\}.$$
Then there is the surjection
 $\Hom(G_F,\,C_3) \to \mathfrak{Gal}(C_3/F)\,;\,\xi \mapsto \closure{F}^{\Ker \xi}$.
We thus obtain the composed map
$$\delta_F:\dual{\gp_H}(F)/\phi\bigl(\gp_H(F)\bigr) \stackrel{\delta}{\longrightarrow}
 H^1(F,\,\gp_H[\phi]) \simeq \Hom(G_F,\,C_3) \longrightarrow \mathfrak{Gal}(C_3/F).$$
Considering locally again, we have the commutative diagram
$$
\begin{CD}
\dual{\gp_H}(F)/\phi\bigl(\gp_H(F)\bigr)	@>{\delta_F}>> \mathfrak{Gal}(C_3/F)\\
@VV{\res_{\prim{p}}}V				@VV{\res_{\prim{p}}}V\\
\dual{\gp_H}(F_{\prim{p}})/\phi\bigl(\gp_H(F_{\prim{p}})\bigr)
  @>{\delta_{F_{\prim{p}}}}>> \mathfrak{Gal}(C_3/F_{\prim{p}}).
\end{CD}
$$
The following gives a slightly different description for the $\phi$-Selmer group.

\begin{proposition}\label{phiSelmer}
\quad
The map $\delta_{F_{\bullet}}$ is explicitly written as
\begin{align*}
\delta_{F_{\bullet}}:\dual{\gp_H}(F_{\bullet})/\phi\bigl(\gp_H(F_{\bullet})\bigr) \quad &\longrightarrow \quad
 \mathfrak{Gal}(C_3/F_{\bullet})\\
P \quad &\longmapsto \quad F_{\bullet}\bigl(\phi^{-1}(P)\bigr).
\end{align*}
Here $\bullet$ stands for $\prim{p}$ or null space.
Moreover it holds that
$$\sharp S^{(\phi)}(\gp_H/F)=2\cdot\sharp\!\lt\{L \in \mathfrak{Gal}(C_3/F) \;\Big|\; \mbox{\parbox{27ex}{\small
 $L/F$ is unramified outside $S_0$\\
$L\!\cdot\!F_{\prim{p}} \in \Image \delta_{F_{\prim{p}}}$ for any $\prim{p} \in S_0$}}\rt\}-1$$
where $S_0$ denotes the set of finite primes of $F$ at which $\gp_H$ has bad reduction or dividing $3$.
\end{proposition}

\vsp

\begin{proof}
\quad
By tracing the definitions and consequences of descent theory.
For details, see \cite{RK}, \S2.
\end{proof}

\begin{remark}\label{K}
\quad
Since $\gp_H(F)[\phi]=\langle \ellunit'^{\sigma}-\ellunit' \rangle$ by the definition,
 the point $P=\phi(\ellunit') \in \dual{\gp_H}$ is $F$-rational.
For any cyclic cubic extension $\gext/F$ and any hyperplane $H \in \mathcal{L}(\ellunit)$ not tangent to $\gp$,
 taking such $P \in \dual{\gp_H}(F)$ yields $\delta_F(P)=F(\ellunit')=\gext$.
Thus all cyclic cubic extensions over $F$ are subject to
 $F$-rational points on the elliptic curves $\bigl\{\dual{\gp_H}/F\bigr\}$.
Furthermore it is also observed that the elliptic curve $\gp_H/F$ has bad reduction
 at each prime $\prim{p} \nmid 3$ which ramifies in $\gext$.
\end{remark}

\subsection{On the Selmer group $S^{(\dual{\phi})}(\dual{\gp_H}/F)$}

As for the dual side,
 the Weil pairing $e_{\dual{\phi}}$ for the isogeny $\dual{\phi}$ constructs the commutative diagram
$$
\begin{CD}
\gp_H(F)/\dual{\phi}\bigl(\dual{\gp_H}(F)\bigr)	@>{\dual{\delta}}>> H^1(F,\,\dual{\gp_H}[\dual{\phi}])
@>{e_{\dual{\phi}}(\,\cdot\,,\,T)}>>	H^1(F,\,\mu_3) \simeq F^{*}/F^{*3}\\
@VVV						@VVV
@VVV\\
\gp_H(\gext)/\dual{\phi}\bigl(\dual{\gp_H}(\gext)\bigr)	@>{\dual{\delta}}>> H^1(\gext,\,\dual{\gp_H}[\dual{\phi}])
@>{e_{\dual{\phi}}(\,\cdot\,,\,T)}>>	H^1(\gext,\,\mu_3) \simeq \gext^{*}/\gext^{*3}.
\end{CD}
$$
We denote by $\dual{\delta}_F,\,\dual{\delta}_F^{\gext}$ the homomorphisms from the group
 $\gp_H(F)/\dual{\phi}\bigl(\dual{\gp_H}(F)\bigr)$ to $F^{*}/F^{*3}$, $\gext^{*}/\gext^{*3}$ respectively.
For each finite prime $\prim{p}$ of $F$, it may be defined similarly the maps
 $\dual{\delta}_{F_\prim{p}}:\gp_H(F_{\prim{p}})/\dual{\phi}\bigl(\dual{\gp_H}(F_{\prim{p}})\bigr)
 \to F_{\prim{p}}^{*}/F_{\prim{p}}^{*3}$ and
 $\dual{\delta}_{F_\prim{p}}^{\gext_{\prim{P}}}:\gp_H(F_{\prim{p}})/\dual{\phi}\bigl(\dual{\gp_H}(F_{\prim{p}})\bigr)
 \to \gext_{\prim{P}}^{*}/\gext_{\prim{P}}^{*3}$ where $\prim{P}$ is a prime of $\gext$ above $\prim{p}$.
There is no need to consider infinite primes since $H^1(F_{\infty},\,\dual{\gp_H}[\dual{\phi}])$ vanishes.
Thus the $\dual{\phi}$-Selmer group can be rewritten as
$$S^{(\dual{\phi})}(\dual{\gp_H}/F)=\lt\{\eqcls{d} \in F^{*}/F^{*3} \;\Big|\;
 \eqcls{d} \in \Image \dual{\delta}_{F_{\prim{p}}} \text{ for any finite prime $\prim{p}$ of $F$}\rt\}.$$

Let $\mathcal{K}_{\prim{p}}$ denote the {\'e}tale algebra $F_{\prim{p}}[x]/\bigl(\mathfrak{g}(x;t)\bigr)$,
 where $\mathfrak{g}$ is the generic polynomial in \S\ref{genpoly},
 which is isomorphic to the direct sum of the completions $\gext_{\prim{P}}$ of $\gext$ at the primes $\prim{P}$
 lying above $\prim{p}$, namely $\mathcal{K}_{\prim{p}} \simeq \prod_{\prim{P}|\prim{p}} \gext_{\prim{P}}
 \simeq \gext \otimes_{F} F_{\prim{p}}$,
 and let $\mathcal{K}_{\prim{p}}^{*}$ denote the multiplicative group
 $\prod_{\prim{P}|\prim{p}} \gext_{\prim{P}}^{*}$.
We frequently use the identification $\mathcal{K}_{\prim{p}}^{*}/\mathcal{K}_{\prim{p}}^{*3}
 \simeq \prod_{\prim{P} \mid \prim{p}} \gext_{\prim{P}}^{*}/\gext_{\prim{P}}^{*3}$.
The natural inclusion $F \subset F_{\prim{p}}$ induces an embedding $\gext \hookrightarrow \mathcal{K}_{\prim{p}}$
 by the identification $\gext=F[x]/\bigl(\mathfrak{g}(x;t)\bigr)$.
We define a map $\dual{\delta}_{F_\prim{p}}^{\mathcal{K}_{\prim{p}}}
:\gp_H(F_{\prim{p}})/\dual{\phi}\bigl(\dual{\gp_H}(F_{\prim{p}})\bigr)
 \to \mathcal{K}_{\prim{p}}^{*}/\mathcal{K}_{\prim{p}}^{*3}$
 by $P \mapsto \prod_{\prim{P}|\prim{p}} \dual{\delta}_{F_\prim{p}}^{\gext_{\prim{P}}}(P)$
 which commutes the diagram
$$
\begin{CD}
\dual{\delta}_F^{\gext}:\gp_H(F)/\dual{\phi}\bigl(\dual{\gp_H}(F)\bigr) @>{\dual{\delta}_F}>> F^{*}/F^{*3}
@>>>	\gext^{*}/\gext^{*3}\\
@VVV				@VVV				@VVV\\
\dual{\delta}_{F_\prim{p}}^{\mathcal{K}_{\prim{p}}}:
\gp_H(F_{\prim{p}})/\dual{\phi}\bigl(\dual{\gp_H}(F_{\prim{p}})\bigr)	@>{\dual{\delta}_{F_\prim{p}}}>>
F_{\prim{p}}^{*}/F_{\prim{p}}^{*3}	@>>>	\mathcal{K}_{\prim{p}}^{*}/\mathcal{K}_{\prim{p}}^{*3}.
\end{CD}
$$
As for the maps $\dual{\delta}_F^{\gext},\,\dual{\delta}_{F_\prim{p}}^{\mathcal{K}_{\prim{p}}}$,
 we will give a certain description in Proposition~\ref{connect} which is essential in \S\ref{sec_SI}.
Note that the homomorphisms $\dual{\delta}_F,\,\dual{\delta}_{F_\prim{p}}$ are injective
 from the non-degeneracy of the Weil pairing $e_{\dual{\phi}}$
 but $\dual{\delta}_F^{\gext},\,\dual{\delta}_{F_\prim{p}}^{\mathcal{K}_{\prim{p}}}$ are not necessarily,
 because the natural maps $F^{*}/F^{*3} \to \gext^{*}/\gext^{*3}$
 and $F_{\prim{p}}^{*}/F_{\prim{p}}^{*3} \to \gext_{\prim{P}}^{*}/\gext_{\prim{P}}^{*3}$
 are not injective in general. See also Lemma~\ref{kernel}.

The following lemma is proven via local consideration by applying the Hasse principle for norm varieties,
 which was originally proven in \cite{Hasse}.

\begin{lemma}\label{bound_norm}
\quad
$S^{(\dual{\phi})}(\dual{\gp_H}/F) \subset \Norm{\gext/F}(\gext^{*})/F^{*3}$.
\end{lemma}

\vsp

\begin{proof}
\quad
For arbitrary $\eqcls{d} \in S^{(\dual{\phi})}(\dual{\gp_H}/F)$,
 it suffices to show $d \in \Norm{\gext_{\prim{P}}/F_{\prim{p}}}(\gext_{\prim{P}}^{*})$ for every prime of $F$
 by Hasse norm theorem.
This is immediate for infinite primes or finite primes which split completely in $\gext$.
We thus assume a finite prime $\prim{p}$ does not split completely in $\gext$.
Then $\Gal(\gext_{\prim{P}}/F_{\prim{p}}) \simeq \Gal(\gext/F)$.
Since $d$ represents an element in $S^{(\dual{\phi})}(\dual{\gp_H}/F)$,
 there is a point $[x,\,y,\,z,\,1] \in \gp_H(F_{\prim{p}})$ and some element $\beta \in \gext_{\prim{P}}^{*}$
 such that $d=\frac{x+y\alpha_1^{\sigma}+z\alpha_2^{\sigma}}{x+y\alpha_1+z\alpha_2}\beta^3$
 from Proposition~\ref{connect}.
Note here that $\gp_H(F_{\prim{p}}) \cap \{w=0\}=\emptyset$ if and only if $F_{\prim{p}}\ne\gext_{\prim{P}}$.
From Lemma~\ref{norm} we have $\Norm{\gext_{\prim{P}}/F_{\prim{p}}}(\beta)=d$.
Gathering all the pieces, we can conclude $d \in \Norm{\gext/F}(\gext^{*})$.
\end{proof}

\begin{remark}
\quad
From the above we have $\Image \dual{\delta}_F \subset \Norm{\gext/F}(\gext^{*})/F^{*3}$.
\end{remark}

\subsection{Combining the Selmer groups $S^{(\phi)}(\gp_H/F)$ and $S^{(\dual{\phi})}(\dual{\gp_H}/F)$}

Given an isogeny and its dual isogeny, it is known Cassels's duality formula in \cite{Cas1}
 which connects order of a Selmer group with that of its dual Selmer group.
Applying the result in \cite{RK} yields the following connection formula.
However, for simplicity, we consider only the case $a=0$ which we will use later, namely the hyperplane $H:by+cz=0$
 with coefficients $[b,\,c] \in \mathbb{P}^1(\ringint{F})$.

\begin{proposition}\label{duality}
\quad
Assume $F$ has class number $1$ in which $3$ splits completely.
Let $v_{\prim{p}}:=\nu_{\prim{p}}\bigl(\Norm{\gext/F}(\gamma(\gp_H))\bigr) \in \rint$ for $\prim{p} \in \primes_F$
 where $\primes_F$ stands for the set of all finite primes of $F$.
For a hyperplane $H:by+cz=0$ which is not tangent to $\gp$, it holds that
$$\dim_{\mathbb{F}_3} S^{(\phi)}(\gp_H/F)=\dim_{\mathbb{F}_3} S^{(\dual{\phi})}(\dual{\gp_H}/F)
+d_{+}-d_{-}+d_{\infty}-r$$
where $r$ denotes the rank of the group $\ringint{F}^*$ of units and
\begin{align*}
d_{+}&:=\sharp\!\lt\{\prim{p} \in \primes_F \;\Big|\; \prim{p} \mid 3 \text{ and }
 \Norm{\gext/F}\bigl(\gamma(\gp_H)\bigr) \in F_{\prim{p}}^{*3}\rt\},\\
d_{-}&:=\sharp\!\lt\{\prim{p} \in \primes_F \;\Big|\; 3 \nmid v_{\prim{p}} \text{ and }
 \bigl\{\prim{p} \mid 3 \text{ or } \zeta_3 \notin F_{\prim{p}}\bigr\}\rt\},\\
d_{\infty}&:=\sharp\!\lt\{\prim{p} \in \primes_F \;\Big|\; v_{\prim{p}} \equiv 2 \!\!\!\!\pmod{3}
 \text{ and } \prim{p} \mid 3\rt\}.
\end{align*}
\end{proposition}

\vsp

\begin{proof}
\quad
This is a direct consequence from \cite{RK} combining with Cassels's duality formula.
\end{proof}

\begin{remark}\label{assumption}
\quad
Assume $3$ is unramified in $F$.
Let $\prim{p}$ be a prime of $F$ above $3$.
Then, for an element $d \in F_{\prim{p}}^{*}$, the condition $d \notin F_{\prim{p}}^{*3}$ is equivalent to
 the condition either $3 \nmid \nu_{\prim{p}}(d)$ or
 $3 \mid \nu_{\prim{p}}(d),\,u_{\pi} \not\equiv k_{\pi}^3 \pmod{\pi^2}$
 for the constants $u_{\pi},\,k_{\pi} \in \ringint{F_{\prim{p}}}^{*}$ which are defined to be
 $u_{\pi} \equiv k_{\pi}^3 \pmod{\pi}$
 with $u_{\pi}:=\pi^{-\nu_{\prim{p}}(d)}d$ where $\pi$ stands for a prime element in $F_{\prim{p}}$.
Further we assume $3$ splits completely in $F$.
Then $u_{\pi}$ is equivalent to $\pm1$ modulo $\pi$ and hence we can choose $k_{\pi}=\pm1$ respectively.
\end{remark}

\section{Unit groups, Class groups and Selmer groups}\label{sec_SI}

As for relations between class groups and Selmer groups, it has been investigated by many authors.
For example, L. C. Washington \cite{Was} discussed the connection of Selmer groups to class groups of
 cyclic cubic fields over $\rat$, N. Aoki \cite{Aoki} of cubic extensions over arbitrary number field
 containing a primitive cube root of unity.
E. F. Schaefer \cite{Sch} also investigated the case when those groups coincide
 for the minimal field of definition of torsion points on abelian varieties.
This section is intended to study the family of elliptic curves $\{\gp_H\}_H$ in such direction.
However, in addition to class groups, here we also take unit groups into account
 in view of the analogy between Mordell-Weil groups and unit groups.

For any set $S$ of finite primes of $F$, and any finite extension $L$ over $F$, define
\begin{align*}
\ringint{L,S}&:=\bigl\{\alpha \in L \;\big|\; \nu_{\prim{P}}(\alpha) \geq 0
 \text{ for any $\prim{P} \in \primes_{L,S}$}\bigr\},\\
\mathfrak{R}_{L,S}&:=\lt\{\eqcls{\alpha} \in L^{*}/L^{*3} \;\Big|\; \nu_{\prim{P}}(\alpha) \equiv 0\!\!\!\!\pmod{3}
 \text{ for any $\prim{P} \in \primes_{L,S}$}\rt\},\\
\iclsgp_{L,S}&:=\bigl\{\text{ non-zero fractional ideals of $\ringint{L,S}$ }\bigr\}/L^{*}\ringint{L,S}.
\end{align*}
Here $\primes_{L,S}$ stands for the set of all finite primes of $L$ outside $S$.
When $S=\emptyset$, we omit the letter $S$ as $\primes_L,\,\ringint{L},\,\mathfrak{R}_{L},\,\iclsgp_{L}$ and so on.
In this section, we show some connection between the Selmer group $S^{(\dual{\phi})}(\dual{\gp_H}/F)$
 and unit groups $\ringint{L,S}^{*}$, class groups $\iclsgp_{L,S}$ in the case $L=F$ or $\gext$.
To begin with, we quote \cite{RK}, Proposition~4.5, which implies the following.

\begin{lemma}\label{loceq}
\quad
The following equivalences hold.
\begin{align*}
\Image \dual{\delta}_{F_{\prim{p}}} \hookrightarrow \ringint{F_{\prim{p}}}^{*}/\ringint{F_{\prim{p}}}^{*3} \quad
 &\Longleftrightarrow \quad \prim{p} \in \primes_{F,S},\\
\Image \dual{\delta}_{F_{\prim{p}}}^{\mathcal{K}_{\prim{p}}} \hookrightarrow
 \ringint{\mathcal{K}_{\prim{p}}}^{*}/\ringint{\mathcal{K}_{\prim{p}}}^{*3} \quad
 &\Longleftrightarrow \quad \prim{p} \in \primes_{F,S'},
\end{align*}
where $\ringint{\mathcal{K}_{\prim{p}}}^{*}/\ringint{\mathcal{K}_{\prim{p}}}^{*3}
=\prod_{\prim{P}|\prim{p}} \ringint{\gext_{\prim{P}}}^{*}/\ringint{\gext_{\prim{P}}}^{*3}$
 and $S,\,S'$ are finite sets defined by
\begin{align*}
S&:=\lt\{\prim{p} \in \primes_F \;\Big|\; \mbox{\parbox{35ex}{\small
 $3\nu_{\prim{p}}\bigl(a\delta(\gp)\bigr)<\nu_{\prim{p}}\bigl(\Norm{\gext/F}(\gamma(\gp_H))\bigr)$\\
 \quad or \; $\nu_{\prim{p}}\bigl(\Norm{\gext/F}(\gamma(\gp_H))\bigr) \not\equiv 0\!\!\!\!\pmod{3}$}}\rt\},\\
S'&:=\lt\{\prim{p} \in S \;\Big|\; \text{$\prim{p}$ splits completely in $\gext$}\rt\}.
\end{align*}
\end{lemma}

\vsp

\begin{proof}
\quad
The first equivalence follows directly from \cite{RK}, Proposition~4.5.
As for the second, for any $\prim{p} \in \primes_{F,S'}$ if $\prim{p}$ is inert or ramifies in $\gext$ then 
 $\Image \dual{\delta}_{F_{\prim{p}}}^{\mathcal{K}_{\prim{p}}}
 \hookrightarrow \ringint{\mathcal{K}_{\prim{p}}}^{*}/\ringint{\mathcal{K}_{\prim{p}}}^{*3}$
 by using Proposition~\ref{connect}.
If $\prim{p}$ splits completely in $\gext$ then $\prim{p} \in \primes_{F,S}$,
 and hence the statement follows from the first equivalence
 since $F_{\prim{p}}=\gext_{\prim{P}}$ as valuation fields for each $\prim{P} \mid \prim{p}$.
Conversely, for any $\prim{p} \in S'$ the first equivalence implies
$\Image \dual{\delta}_{F_{\prim{p}}}^{\mathcal{K}_{\prim{p}}} \not\hookrightarrow
 \ringint{\mathcal{K}_{\prim{p}}}^{*}/\ringint{\mathcal{K}_{\prim{p}}}^{*3}$.
\end{proof}

By the above lemma, it turns out that each element in $S^{(\dual{\phi})}(\dual{\gp_H}/F)$
 is of $\prim{P}$-adic valuation divided by $3$ outside $S'$, namely $S^{(\dual{\phi})}(\dual{\gp_H}/F)$
 relates to the $3$-torsion subgroup of the $S'$-ideal class group $\iclsgp_{\gext,S'}$.
We will also consider the relation to that of the $S$-ideal class group $\iclsgp_{F,S}$.

We make notation as
\begin{align*}
{_\mathrm{N}\Phi_{F,S}}&:=\frac{\Norm{\gext/F}(\gext^{*}) \cap \ringint{F,S}^{*}}{\ringint{F,S}^{*3}}
\;\bigl(\subset \ringint{F,S}^{*}/\ringint{F,S}^{*3}\bigr),\\
{_\mathrm{N}\iclsgp_{F,S}[3]}
&:=\lt\{\eqcls{\ideal{a}} \in \iclsgp_{F,S}[3] \;\Big|\; \ideal{a}^3=d\ringint{F,S},\,
d \in \Norm{\gext/F}(\gext^{*})\rt\}\;\bigl(\subset \iclsgp_{F,S}[3]\bigr),\\
{_\mathrm{N}\mathfrak{U}_{\gext,S'}}
&:=\lt\{\eqcls{\alpha^{\sigma-1}} \in \ringint{\gext,S'}^{*}\gext^{*3}/\gext^{*3}
 \;\Big|\; \alpha \in \gext^{*},\,\Norm{\gext/F}(\alpha)=1\rt\}
\;\bigl(\subset \ringint{\gext,S'}^{*}/\ringint{\gext,S'}^{*3}\bigr),\\
{_\mathrm{N}\iclsgp_{\gext,S'}[3]}
&:=\lt\{\eqcls{\ideal{A}} \in \iclsgp_{\gext,S'}[3] \;\Big|\; \ideal{A}^3=\alpha\ringint{\gext,S'},\,
\Norm{\gext/F}(\alpha)=1\rt\}\;\bigl(\subset \iclsgp_{\gext,S'}[3]\bigr).
\end{align*}
As is indicated, these groups are finite and subject to certain norm condition.
Note that the Galois group $G=\Gal(\gext/F)$ acts canonically on the quotient $\gext^{*}/\gext^{*3}$
 and the $S'$-ideal class group $\iclsgp_{\gext,S'}$ by $\sigma(\eqcls{\alpha})=\eqcls{\alpha^{\sigma}}$,
 $\sigma(\eqcls{\ideal{A}})=\eqcls{\ideal{A}^{\sigma}}$ respectively.
The purpose of this section is to prove the following results.

\begin{theorem}\label{bound}
\quad
For any hyperplane $H/F \in \mathcal{L}(\ellunit)$ which is not tangent to $\gp$,
 let $S,\,S'$ be the finite sets in Lemma~\ref{loceq}. Then there is a finite bound
\begin{align*}
\dim_{\mathbb{F}_3} S^{(\dual{\phi})}(\dual{\gp_H}/F)
&\leq \dim_{\mathbb{F}_3} {_\mathrm{N}\Phi_{F,S}}+\dim_{\mathbb{F}_3} {_\mathrm{N}\iclsgp_{F,S}[3]}\\
&\leq \dim_{\mathbb{F}_3} {_\mathrm{N}\mathfrak{U}_{\gext,S'}} \oplus \mu_3(F)
+\dim_{\mathbb{F}_3} {_\mathrm{N}\iclsgp_{\gext,S'}^G[3]}-\dim_{\mathbb{F}_3} \text{\footnotesize
 $\!\frac{\Norm{\gext/F}(\gext_{S'}) \cap \ringint{F,S'}^{*}}{\Norm{\gext/F}(\ringint{\gext,S'}^{*})}$}
\end{align*}
where $\gext_{S'}:=\lt\{\beta \in \gext^{*} \;|\; \nu_{\prim{P}}(\beta^{\sigma-1}) \equiv 0 \pmod{3}
 \text{ for any $\prim{P} \in \primes_{\gext,S'}$}\rt\}$.\\
Furthermore if $\iclsgp_{\gext,S'}$ is $G$-invariant then
 $\gp_H(F)/\dual{\phi}\bigl(\dual{\gp_H}(F)\bigr) \hookrightarrow
 {_\mathrm{N}\mathfrak{U}_{\gext,S'}} \oplus \mu_3(F)$.
\end{theorem}

\begin{theorem}\label{equality'}
\quad
Let $F$ be a number field of class number $1$ in which $3$ splits completely.
For a cyclic cubic extension $\gext/F$, let $\{1,\,\alpha,\,\alpha^2\}$ be a $F$-basis of $\gext$
 that $\gp$ is defined by the polynomial \eqref{defpoly} with.
Here $\alpha \in \gext$ denotes a root of the generic polynomial $\mathfrak{g}(x;t)$
 in \S\ref{genpoly} for some $t \in F$.
Let $H$ be the hyperplane $z=0$.
Assume $t \in \ringint{F}$, $\delta(\mathfrak{g}) \notin F_{\prim{p}}^{*3}$
 for each $\prim{p} \mid 3$.
Then the set $S$ in Lemma~\ref{loceq} consists of all the finite primes of $F$ which ramifies in $\gext$ and
$$\dim_{\mathbb{F}_3} S^{(\phi)}(\gp_H/F)+\rank_{\rint} \ringint{F}^{*}
=\dim_{\mathbb{F}_3} S^{(\dual{\phi})}(\dual{\gp_H}/F)=\dim_{\mathbb{F}_3} {_\mathrm{N}\Phi_{F,S}},$$
$$1 \leq \rank_{\rint} \gp_H(F) \leq 2\dim_{\mathbb{F}_3} {_\mathrm{N}\Phi_{F,S}}-\rank_{\rint} \ringint{F}^{*}-1.$$
\end{theorem}

Combining Proposition~\ref{phiSelmer} with Theorem~\ref{equality} and \ref{equality'}
 yields a certain description for ${_\mathrm{N}\iclsgp_{\gext}^G[3]}$.

\begin{corollary}\label{cl}
\quad
Under the first assumption in Theorem~\ref{equality}, there is the equality
\begin{align*}
3\cdot\sharp{_\mathrm{N}\iclsgp_{\gext}^G[3]}
&=2\cdot\sharp\!\lt\{L \in \mathfrak{Gal}(C_3/\rat) \;\Big|\; \mbox{\parbox{27ex}{\small
 $L/\rat$ is unramified outside $S_0$\\
$L\!\cdot\!\rat_p \in \Image \delta_{\rat_p}$ for any $p \in S_0$}}\rt\}-1\\
&=\sharp{_\mathrm{N}\Phi_{\rat,S}}
\end{align*}
where $S$ is the set of all the finite primes of $F$ which ramifies in $\gext$
 and the other notation is the same as in Proposition~\ref{phiSelmer}.
\end{corollary}

\begin{example}
\quad
Here we give an example for ${_\mathrm{N}\iclsgp_{\gext}^G[3]}$ being non-trivial in the case $F=\rat$.
Let $t=-3^3$. Then the conductor of the minimal splitting field of $\mathfrak{g}(x;t)$ is $657=3^2 \cdot 73$,
 and the class number is $9$ (see \cite{Gras}).
Since both $3,\,73$ lies in the image of the norm map,
 Using Corollary~\ref{cl} yields ${_\mathrm{N}\iclsgp_{\gext}^G[3]} \simeq \rint/3\rint$.
This example also satisfy $\delta(\mathfrak{g})=657 \notin \rat_3^{*3}$ (cf. Theorem~\ref{equality}).
\end{example}

The rest of this section is devoted to the proof of the above theorems.
First of all, in view of Lemma~\ref{bound_norm}, \ref{loceq},
 we shall introduce finite groups for the sets $S,\,S'$ in Lemma~\ref{loceq} defined by
\begin{align*}
{_\mathrm{N}\mathfrak{R}_{F,S}}:=\lt\{\eqcls{d} \in \mathfrak{R}_{F,S} \;\Big|\;
 d \in \Norm{\gext/F}(\gext^{*})\rt\}, \quad
{_\mathrm{N}\mathfrak{R}_{\gext,S'}}:=\lt\{\eqcls{\alpha} \in \mathfrak{R}_{\gext,S'} \;\Big|\;
 \Norm{\gext/F}(\alpha)=1\rt\}.
\end{align*}
From Lemma~\ref{G-inv}, the $G$-invariant subgroup ${_\mathrm{N}\mathfrak{R}_{\gext,S'}^G}$ may be written by
$${_\mathrm{N}\mathfrak{R}_{\gext,S'}^G}=\lt\{\eqcls{d} \in \mu_3(F)\Norm{\gext/F}(\gext^{*})\gext^{*3}/\gext^{*3}
 \;\Big|\; \nu_{\prim{P}}(d) \equiv 0\!\!\!\!\pmod{3} \text{ for any $\prim{P} \in \primes_{\gext,S'}$}\rt\}.$$
When $\zeta_3 \in F$, the cyclic cubic extension $\gext$ is generated over $F$ by $\sqrt[3]{D}$
 for some $D \in F^{*} \setminus F^{*3}$ from Kummer theory,
 which implies $D=\Norm{\gext/F}(\sqrt[3]{D})$ and $\nu_{\prim{P}}(D) \equiv 0 \pmod{3}$ for any
 $\prim{P} \in \primes_{\gext}$ since each prime $\prim{p} \in \primes_F$ with $3 \nmid \nu_{\prim{p}}(D)$
 ramifies in $\gext$. Combining this with Lemma~\ref{kernel} yields
\begin{align*}
{_\mathrm{N}\mathfrak{R}_{\gext,S'}^G} &\oplus (F^{*} \cap \gext^{*3})/F^{*3}\\
&\simeq \lt\{\eqcls{d} \in \mu_3(F)\Norm{\gext/F}(\gext^{*})/F^{*3} \;\Big|\; \nu_{\prim{P}}(d) \equiv
 0\!\!\!\!\pmod{3} \text{ for any $\prim{P} \in \primes_{\gext,S'}$}\rt\}\\
&\simeq \lt\{\eqcls{d} \in \Norm{\gext/F}(\gext^{*})/F^{*3} \;\Big|\; \nu_{\prim{P}}(d) \equiv
 0\!\!\!\!\pmod{3} \text{ for any $\prim{P} \in \primes_{\gext,S'}$}\rt\} \oplus \eqcls{\mu_3}(F)\\
&=\lt\{\eqcls{d} \in \Norm{\gext/F}(\gext^{*})/F^{*3} \;\Big|\; \nu_{\prim{p}}(d)
 \equiv 0\!\!\!\!\pmod{3} \text{ for any $\prim{p} \in \primes_{F,S \cup R}$}\rt\} \oplus \eqcls{\mu_3}(F)
\end{align*}
where $\eqcls{\mu_3}(F):=\frac{\mu_3(F)\Norm{\gext/F}(\gext^{*})}{\Norm{\gext/F}(\gext^{*})}$ and
 $R$ stands for the set of all the finite primes of $F$ which ramifies in $\gext$.
We understand $\oplus$ as the direct sum of $\mathbb{F}_3$-vector spaces.
The last equality follows from the definition of $S,\,S'$.
It thus turns out that ${_\mathrm{N}\mathfrak{R}_{F,S}} \oplus \eqcls{\mu_3}(F) \hookrightarrow
 {_\mathrm{N}\mathfrak{R}_{\gext,S'}^G} \oplus (F^{*} \cap \gext^{*3})/F^{*3}$.
Then these groups contain the $\dual{\phi}$-Selmer group as follows.

\begin{proposition}\label{bound_R}
\quad
There are inclusions as $\mathbb{F}_3$-vector spaces$:$
$$S^{(\dual{\phi})}(\dual{\gp_H}/F) \oplus \eqcls{\mu_3}(F) \;\hookrightarrow\;
 {_\mathrm{N}\mathfrak{R}_{F,S}} \oplus \eqcls{\mu_3}(F) \;\hookrightarrow\;
 {_\mathrm{N}\mathfrak{R}_{\gext,S'}^G} \oplus (F^{*} \cap \gext^{*3})/F^{*3}.$$
Furthermore, if we assume the following conditions then all those groups coincide and $S=R,\,S'=\emptyset$.
Let $v_{\prim{p}}:=\nu_{\prim{p}}\bigl(\Norm{\gext/F}(\gamma(\gp_H))\bigr) \in \rint$ for $\prim{p} \in \primes_F$.
\begin{list}{$\diamond$}{}
\item $3$ splits completely in $F$,
\item The hyperplane is $H:z=0$,
\item Each finite prime $\prim{p} \in \primes_F$
 satisfying $3 \nmid v_{\prim{p}}$ ramifies in $\gext$,
\item For each $\prim{p} \mid 3$, if $3 \mid v_{\prim{p}}$ then
 $\prim{p}$ is unramified in $\gext$ and $\Norm{\gext/F}\bigl(\gamma(\gp_H)\bigr) \notin F_{\prim{p}}^{*3}$.
\end{list}
\end{proposition}

\vsp

\begin{proof}
\quad
Using Lemma~\ref{bound_norm} and \ref{loceq} yields
 $S^{(\dual{\phi})}(\dual{\gp_H}/F) \subset {_\mathrm{N}\mathfrak{R}_{F,S}}$.
Now we shall prove the latter statement.
Since $H:z=0$, a direct calculation shows us
 $\Norm{\gext/F}\bigl(\gamma(\gp_H)\bigr)=-\Norm{\gext/F}(\alpha_1-\alpha_1^{\sigma})$
 where $\{1,\,\alpha_1,\,\alpha_2\}$ denotes the $F$-basis of $\gext$ for $\gp$.
It thus follows from Lemma~\ref{tame} that $S=R$, $S'=\emptyset$ under the conditions,
 and in particular ${_\mathrm{N}\mathfrak{R}_{F,S}} \oplus \eqcls{\mu_3}(F) \simeq
 {_\mathrm{N}\mathfrak{R}_{\gext,S'}^G} \oplus (F^{*} \cap \gext^{*3})/F^{*3}$.
Finally we show the inclusion $S^{(\dual{\phi})}(\dual{\gp_H}/F) \supset {_\mathrm{N}\mathfrak{R}_{F,S}}$.
Let $\eqcls{d}$ be arbitrary element in ${_\mathrm{N}\mathfrak{R}_{F,S}}$.
For any $\prim{p} \in \primes_F$, the following chart can be verified from \cite{RK}, Theorem~4.1.
$$
\begin{cases}
3 \mid v_{\prim{p}} \;\rightarrow\;
\begin{cases}
\prim{p} \mid 3 \;\rightarrow\; \text{See below $\cdots (*)$}.\\
\prim{p} \nmid 3 \;\rightarrow\; \prim{p} \in \primes_{F,S} \text{ and hence
 $\eqcls{d} \in \Image \dual{\delta}_{F_{\prim{p}}}=\ringint{F_{\prim{p}}}^{*}/\ringint{F_{\prim{p}}}^{*3}$}.
\end{cases}\\
3 \nmid v_{\prim{p}} \;\rightarrow\;
 \text{From Lemma~\ref{coincide}, $\eqcls{d} \in \Image \dual{\delta}_{F_{\prim{p}}}
=\Norm{\gext_{\prim{P}}/F_{\prim{p}}}(\gext_{\prim{P}}^{*})/F_{\prim{p}}^{*3}$}.
\end{cases}
$$
We consider the lack of the above chart indicated $(*)$.
By using the equivalence and the notation in Remark~\ref{assumption},
 the condition $3 \mid v_{\prim{p}},\,\Norm{\gext/F}\bigl(\gamma(\gp_H)\bigr) \notin F_{\prim{p}}^{*3}$
 can be read into $u_{\pi} \not\equiv k_{\pi}^3 \pmod{9}$.
Here we may choose $3$ as a prime element $\pi \in F_{\prim{p}}$.
Let $\mathscr{D}':=u_{\pi}\bigl(u_{\pi}+\frac{u_{\pi}-k_{\pi}^3}{3}\bigr)$.
Since $u_{\pi} \equiv \pm1 \pmod{3}$ from the condition $3$ splits completely in $F$,
 it can be verified that $u_{\pi} \not\equiv k_{\pi}^3 \pmod{9}$
 if and only if $\bigl(\frac{\mathscr{D}'}{\prim{p}}\bigr) \ne 1$, which denotes the Legendre symbol.
From \cite{RK}, Theorem~4.1, this implies $\eqcls{d} \in \Image \dual{\delta}_{F_{\prim{p}}}$.
Consequently, we have $\eqcls{d} \in S^{(\dual{\phi})}(\dual{\gp_H}/F)$.
This completes the proof.
\end{proof}

\begin{remark}
\quad
We consider the assumption in the above proposition more explicitly.
Let $\alpha \in \closure{F}$ be a root of the generic polynomial $\mathfrak{g}(x;t)$\,$($see \S\ref{genpoly}$)$
 with a fixed $t \in F$ so as to be $\gext=F(\alpha)$.
We choose $\{1,\,\alpha,\,\alpha^2\}$ as the $F$-basis of $\gext$ for $\gp$.
For any $\prim{p} \in \primes_F$ with $\nu_{\prim{p}}(t)<0$ we assume $3 \mid \nu_{\prim{p}}(t)$.
Let $H$ be the hyperplane $z=0$.
Then $\pm\Norm{\gext/F}\bigl(\gamma(\gp_H)\bigr)=\delta(\mathfrak{g})=t^2+3t+9$
 satisfies the third condition by Proposition~\ref{conductor}.
Under the assumption that $3$ splits completely in $F$,
 we can describe the condition
 $3 \mid v_{\prim{p}},\,\delta(\mathfrak{g}) \notin F_{\prim{p}}^{*3}$ for $\prim{p} \mid 3$
 which is equivalent to $u_{\pi} \not\equiv k_{\pi}^3 \pmod{9}$ by using words of $t$ as follows.
 $\bigl(\text{By Proposition~\ref{conductor}, the condition implies that $\prim{p}$ is unramified in $\gext$.}\bigr)$
$$\text{\small
$\begin{cases}
\begin{array}{@{\hspace{0mm}}l@{\hspace{1mm}}l@{\hspace{1mm}}l}
\nu_{\prim{p}}(t)<0 &\rightarrow v_{\prim{p}}=2\nu_{\prim{p}}(t) &\rightarrow
 \text{\rm The condition is equal to $3^{-\nu_{\prim{p}}(t)}t \equiv \pm2,\,\pm4 \!\!\!\!\pmod{9}$}.\\
\nu_{\prim{p}}(t)=0 &\rightarrow v_{\prim{p}}=0 &\rightarrow
 \text{\rm The condition is equal to $t \equiv 1,\,5,\,7,\,8 \!\!\!\!\pmod{9}$}.\\
t \equiv 0,\,6 \!\!\!\!\pmod{9} &\rightarrow v_{\prim{p}}=2 &\rightarrow
 \text{This does not conform to $3 \mid v_{\prim{p}}$, which is ignored}.\\
t \equiv 3 \!\!\!\!\pmod{9} &\rightarrow v_{\prim{p}}=3 &\rightarrow
 \text{\rm The condition is equal to $t \equiv 12 \!\!\!\!\pmod{27}$}.
\end{array}
\end{cases}$}
$$
This is often convenient for an explicit calculation.
\end{remark}

The group $\mathfrak{R}_{F,S}$ relates to the $3$-torsion subgroup of the $S$-ideal class group
 $\iclsgp_{F,S}$ and the group of $S$-units $\ringint{F,S}^{*}$.
There is an exact sequence
$$0 \longrightarrow \ringint{F,S}^{*}F^{*3}/F^{*3} \longrightarrow \mathfrak{R}_{F,S}
 \longrightarrow \iclsgp_{F,S}[3] \longrightarrow 0$$
which is induced by the surjective homomorphism
\begin{align*}
\mathfrak{R}_{F,S} \quad &\longrightarrow \quad \iclsgp_{F,S}[3]\\
\eqcls{d} \quad &\longmapsto \quad \eqcls{\prod_{\prim{p} \in \primes_{F,S}}
 \ideal{p}^{\nu_{\prim{p}}(d)/3}\ringint{F,S}}.
\end{align*}
Since ${_\mathrm{N}\mathfrak{R}_{F,S}}$ is a subgroup of $\mathfrak{R}_{F,S}$,
 one can obtain a restricted map from ${_\mathrm{N}\mathfrak{R}_{F,S}}$ to $\iclsgp_{F,S}[3]$.
According to this,
 we would like to restrict the group $\iclsgp_{F,S}[3]$ so as to hold the surjectivity of this map.
This is the reason why we introduced the subgroup ${_\mathrm{N}\iclsgp_{F,S}[3]}$.
It is easily verified that the group ${_\mathrm{N}\iclsgp_{F,S}[3]}$
 is the exact image of ${_\mathrm{N}\mathfrak{R}_{F,S}}$.
We thus obtain the exact sequence
\begin{align}\label{seq_F}
0 \longrightarrow {_\mathrm{N}\Phi_{F,S}} \longrightarrow {_\mathrm{N}\mathfrak{R}_{F,S}}
 \longrightarrow {_\mathrm{N}\iclsgp_{F,S}[3]} \longrightarrow 0
\end{align}
where ${_\mathrm{N}\Phi_{F,S}}=\bigl\{\eqcls{d} \in \ringint{F,S}^{*}F^{*3}/F^{*3} \;\big|\;
 d \in F^{*},\,d \in \Norm{\gext/F}(\gext^{*})\bigr\}$ is the kernel of the map.

Similarly, the group $\mathfrak{R}_{\gext,S'}$ also relates to the $3$-torsion subgroup of the $S'$-ideal class group
 $\iclsgp_{\gext,S'}$ and the group of $S'$-units $\ringint{\gext,S'}^{*}$.
The restriction to the subgroup ${_\mathrm{N}\mathfrak{R}_{\gext,S'}}$ of the surjective homomorphism as $G$-modules
\begin{align*}
\mathfrak{R}_{\gext,S'} \quad &\longrightarrow \quad \iclsgp_{\gext,S'}[3]\\
\eqcls{\alpha} \quad &\longmapsto \quad \eqcls{\prod_{\prim{P} \in \primes_{\gext,S'}}
 \ideal{P}^{\nu_{\prim{P}}(\alpha)/3}\ringint{\gext,S'}}
\end{align*}
induces the exact sequence
$$0 \longrightarrow \Psi \longrightarrow {_\mathrm{N}\mathfrak{R}_{\gext,S'}}
 \longrightarrow {_\mathrm{N}\iclsgp_{\gext,S'}[3]} \longrightarrow 0$$
with the kernel $\Psi:=\bigl\{\eqcls{\alpha} \in \ringint{\gext,S'}^{*}\gext^{*3}/\gext^{*3} \;\big|\;
 \alpha \in \gext^{*},\,\Norm{\gext/F}(\alpha)=1\bigr\}$, which can be rewritten as
$$\Psi=\lt\{\eqcls{\epsilon} \in \ringint{\gext,S'}^{*}\gext^{*3}/\gext^{*3}
 \;\Big|\; \epsilon \in \ringint{\gext,S'}^{*},\,\Norm{\gext/F}(\epsilon) \in \Norm{\gext/F}(\gext^{*3})\rt\}.$$
Taking Galois cohomology yields the long exact sequence
$$0 \longrightarrow \Psi^G \longrightarrow {_\mathrm{N}\mathfrak{R}_{\gext,S'}^G} \longrightarrow
 {_\mathrm{N}\iclsgp_{\gext,S'}^G[3]} \longrightarrow H^1(\gext/F,\,\Psi) \longrightarrow
 H^1(\gext/F,\,{_\mathrm{N}\mathfrak{R}_{\gext,S'}}) \longrightarrow \cdots.$$
We can describe the image of ${_\mathrm{N}\iclsgp_{\gext,S'}^G[3]}$ in $H^1(\gext/F,\,\Psi)$ as follows.

\begin{proposition}\label{exact}
\quad
The following sequence is exact.
$$0 \longrightarrow \Psi^G \longrightarrow {_\mathrm{N}\mathfrak{R}_{\gext,S'}^G} \longrightarrow
 {_\mathrm{N}\iclsgp_{\gext,S'}^G[3]} \longrightarrow
 \frac{\Norm{\gext/F}(\gext_{S'}) \cap \ringint{F,S'}^{*}}{\Norm{\gext/F}(\ringint{\gext,S'}^{*})}
 \longrightarrow 0.$$
Here $\gext_{S'}:=\lt\{\beta \in \gext^{*} \;|\; \nu_{\prim{P}}(\beta^{\sigma-1}) \equiv 0 \pmod{3}
 \text{ for any $\prim{P} \in \primes_{\gext,S'}$}\rt\}$.
\end{proposition}

\vsp

\begin{proof}
\quad
We have only to prove the following map is surjective with the kernel $\Norm{\gext/F}(\ringint{\gext,S'}^{*})$.
\begin{align*}
\Norm{\gext/F}(\gext_{S'}) \cap \ringint{F,S'}^{*} \quad &\longrightarrow \quad
 \Ker \lt[H^1(\gext/F,\,\Psi) \to H^1(\gext/F,\,{_\mathrm{N}\mathfrak{R}_{\gext,S'}})\rt]\\
\Norm{\gext/F}(\beta) \quad &\longmapsto \quad \eqcls{\xi}
\end{align*}
where $\xi$ denotes a cocycle determined by letting $\xi_{\sigma}:=\eqcls{\Norm{\gext/F}(\beta)}$.
First of all, we show the map is well-defined. If so, this is a homomorphism by definition.
Let $\Norm{\gext/F}(\beta)$ be arbitrary element in $\Norm{\gext/F}(\gext_{S'}) \cap \ringint{F,S'}^{*}$.
Then $\Norm{\gext/F}(\beta)=(\beta^{\sigma-1})^{\sigma-1}\beta^{3\sigma} \in \ringint{F,S'}^{*}$
 and $\Norm{\gext/F}\bigl(\Norm{\gext/F}(\beta)\bigr)=\Norm{\gext/F}(\beta^3) \in \Norm{\gext/F}(\gext^{*3})$,
 therefore $\eqcls{\Norm{\gext/F}(\beta)} \in \Psi$.
Moreover, $\eqcls{\beta^{\sigma-1}} \in {_\mathrm{N}\mathfrak{R}_{\gext,S'}}$ since $\beta \in \gext_{S'}$.
These imply that $\eqcls{\xi}$ lies in the right-hand side.\\
{\bf [Kernel]}
Assume the image $\eqcls{\xi}$ of $\Norm{\gext/F}(\beta) \in \Norm{\gext/F}(\gext_{S'}) \cap \ringint{F,S'}^{*}$
 is trivial, namely $\xi_{\sigma}=\eqcls{\Norm{\gext/F}(\beta)}=\eqcls{\epsilon^{\sigma-1}}$
 for some $\eqcls{\epsilon} \in \Psi$ where $\epsilon \in \ringint{\gext,S'}^{*}$.
Then there is some $\epsilon_0 \in \ringint{\gext,S'}^{*}$ so that
 $\Norm{\gext/F}(\beta)=\epsilon^{\sigma-1}\epsilon_0^3$
 and hence $\Norm{\gext/F}(\beta) \in \Norm{\gext/F}(\ringint{\gext,S'}^{*})$ by Lemma~\ref{norm}.
Conversely, for any $\Norm{\gext/F}(\epsilon) \in \Norm{\gext/F}(\ringint{\gext,S'}^{*})$,
 we find that $\eqcls{\Norm{\gext/F}(\epsilon)}=\eqcls{(\epsilon^{\sigma-1})^{\sigma-1}}$
 corresponds to a coboundary class in $H^1(\gext/F,\,\Psi)$.
Thus the kernel is $\Norm{\gext/F}(\ringint{\gext,S'}^{*})$.\\
{\bf [Surjectivity]} Let $\eqcls{\xi}$ be any element in the right-hand side.
This can be written as $\xi_{\sigma}=\eqcls{\alpha^{\sigma-1}} \in \Psi$ with
 $\eqcls{\alpha} \in {_\mathrm{N}\mathfrak{R}_{\gext,S'}},\,\Norm{\gext/F}(\alpha)=1$ by definition.
From this, we observe $\epsilon:=\alpha^{\sigma-1}\beta^3 \in \ringint{\gext,S'}^{*}$ for some $\beta \in \gext^{*}$,
 and $\epsilon^{\sigma-1}=(\alpha^{-\sigma}\beta^{\sigma-1})^3 \in \ringint{\gext,S'}^{*3}$.
These imply $\Norm{\gext/F}(\beta)^3=\Norm{\gext/F}(\epsilon) \in \ringint{F,S'}^{*}$
 and $\beta \in \gext_{S'}$, namely $\Norm{\gext/F}(\beta) \in \Norm{\gext/F}(\gext_{S'}) \cap \ringint{F,S'}^{*}$.
We claim that this $\Norm{\gext/F}(\beta)$ is mapped to $\eqcls{\xi}$.
Since $\epsilon=\Norm{\gext/F}(\beta)(\alpha^{-\sigma}\beta^{\sigma-1})^{\sigma^2-1}$,
 it is seen that $\eqcls{\Norm{\gext/F}(\beta)}$ is equivalent to $\eqcls{\alpha^{\sigma-1}}$
 modulo $\Psi^{\sigma^2-1}$, which implies the image of $\Norm{\gext/F}(\beta)$ is $\eqcls{\xi}$.
This completes the proof.
\end{proof}

\begin{remark}\label{vanish}
\quad
Assume $F=\rat$ and each prime $\prim{P}$ of $\gext$ lying above any $p \in S'$ is principal.
Since all the primes in $S'$ split completely in $\gext$, there is a prime element $\Pi \in \ringint{\gext,S'}^{*}$
 such that $p=\Norm{\gext/F}(\Pi) \in \ringint{\rat,S'}^{*}$.
We thus have $\ringint{\rat,S'}^{*}=\rint^{*} \oplus \langle p \mid p \in S' \rangle
 \subset \Norm{\gext/\rat}(\ringint{\gext,S'}^{*})$.
In addition, $\iclsgp_{\gext,S'} \simeq \iclsgp_{\gext}/\langle \eqcls{\prim{P}} \rangle_{\prim{P} \mid p \in S'}
 \simeq \iclsgp_{\gext}$ from the property of $S'$-ideal class groups.
Applying these to Proposition~\ref{exact} yields the short exact sequence
$$0 \longrightarrow \Psi^G \longrightarrow {_\mathrm{N}\mathfrak{R}_{\gext,S'}^G} \longrightarrow
 {_\mathrm{N}\iclsgp_{\gext}^G[3]} \longrightarrow 0.$$
\end{remark}

\begin{lemma}\label{kerG}
\quad
$\Psi^G \simeq {_\mathrm{N}\mathfrak{U}_{\gext,S'}} \oplus \eqcls{\mu_3}(F).$
\end{lemma}

\vsp

\begin{proof}
\quad
First, for any $\eqcls{\alpha^{\sigma-1}} \in {_\mathrm{N}\mathfrak{U}_{\gext,S'}}$
 it is seen that $(\alpha^{\sigma-1})^{\sigma-1}=\Norm{\gext/F}(\alpha)\alpha^{-3\sigma} \in \gext^{*3}$
 and $\mu_3(F)\gext^{*3}/\gext^{*3} \subset \Psi^G$.
Thus the inclusion $\mu_3(F){_\mathrm{N}\mathfrak{U}_{\gext,S'}} \subset \Psi^G$ holds.
Conversely, for any $\eqcls{\alpha} \in \Psi^G$ we may assume $\Norm{\gext/F}(\alpha)=1$.
Since $\eqcls{\alpha}$ is $G$-invariant,
 there is some $\beta \in \gext^{*}$ so that $\alpha^{\sigma-1}=\beta^{-3}$,
 hence $\Norm{\gext/F}(\beta)=1$ by using Lemma~\ref{norm}.
From the equality $1=\alpha^{1+\sigma+\sigma^2}=(\alpha^{\sigma}\beta^{1-\sigma})^3$,
 we have $r:=\alpha^{\sigma}\beta^{1-\sigma} \in \mu_3(\gext)=\mu_3(F)$
 and $\eqcls{\alpha}=\eqcls{r\beta^{\sigma-1}} \in \mu_3(F){_\mathrm{N}\mathfrak{U}_{\gext,S'}}$,
 thus $\Psi^G \subset \mu_3(F){_\mathrm{N}\mathfrak{U}_{\gext,S'}}$.
Therefore $\Psi^G=\mu_3(F){_\mathrm{N}\mathfrak{U}_{\gext,S'}}$.
If $\mu_3(F) \subset \Norm{\gext/F}(\gext^{*})$ then $\mu_3(F) \subset {_\mathrm{N}\mathfrak{U}_{\gext,S'}}$
 from Lemma~\ref{norm}, and hence $\Psi^G={_\mathrm{N}\mathfrak{U}_{\gext,S'}}$.
Assume $\mu_3(F) \not\subset \Norm{\gext/F}(\gext^{*})$,
 namely $\mu_3(F)=\langle \zeta_3 \rangle$ and $\mu_3(F) \cap \Norm{\gext/F}(\gext^{*})=\{1\}$.
We claim that the $\mathbb{F}_3$-vector space $\mu_3(F){_\mathrm{N}\mathfrak{U}_{\gext,S'}}$ is isomorphic to
 ${_\mathrm{N}\mathfrak{U}_{\gext,S'}} \oplus \eqcls{\mu_3}(F)$,
 which is given by the map $\eqcls{r\beta^{\sigma-1}} \mapsto (\eqcls{\beta^{\sigma-1}},\,\eqcls{r})$
 where $r \in \mu_3(F),\,\beta^{\sigma-1} \in {_\mathrm{N}\mathfrak{U}_{\gext,S'}}$.
If the map is well-defined then this is an injective homomorphism. The surjectivity is also clear.
For any representative $r'\beta'^{\sigma-1} \in \eqcls{r\beta^{\sigma-1}}$,
 since $r'r^{-1}(\beta'\beta^{-1})^{\sigma-1} \in \gext^{*3}$ and $\Norm{\gext/F}(\beta'\beta^{-1})=1$,
 using Lemma~\ref{norm} yields $r'r^{-1} \in \mu_3(F) \cap \Norm{\gext/F}(\gext^{*})$.
Consequently it must be $r'=r$ and $(\beta'\beta^{-1})^{\sigma-1} \in \gext^{*3}$.
Therefore the map is well-defined.
\end{proof}

Roughly speaking, the following lemma implies that the half of the Mordell-Weil group is bounded by a unit group.

\begin{lemma}\label{embed}
\quad
If $\iclsgp_{\gext,S'}$ is $G$-invariant then $\gp_H(F)/\dual{\phi}\bigl(\dual{\gp_H}(F)\bigr) \hookrightarrow
 {_\mathrm{N}\mathfrak{U}_{\gext,S'}} \oplus (F^{*} \cap \gext^{*3})/F^{*3}$.
\end{lemma}

\vsp

\begin{proof}
\quad
It follows from Lemma~\ref{kernel} that
 $\Image \dual{\delta}_F \hookrightarrow \Image \dual{\delta}_F^{\gext} \oplus (F^{*} \cap \gext^{*3})/F^{*3}$.
Now we shall prove the inclusion
 $\Image \dual{\delta}_F^{\gext} \hookrightarrow {_\mathrm{N}\mathfrak{U}_{\gext,S'}}$.
From Proposition~\ref{connect}, the image of $\dual{\delta}_F^{\gext}$ can be written by
 $\frac{(x+y\alpha_1+z\alpha_2)^{\sigma}}{x+y\alpha_1+z\alpha_2}\gext^{*3}$.
Let $\alpha:=x+y\alpha_1+z\alpha_2$ for short.
Then $\Norm{\gext/F}(\alpha)=1$ by the defining equation of $\gp_H$.
We must show that $\alpha^{\sigma-1}\gext^{*3}$ lies in $\ringint{\gext,S'}^{*}\gext^{*3}/\gext^{*3}$
 under the assumption.
For any $\prim{P} \in \primes_{\gext,S'}$, let $\prim{p}:=\prim{P} \cap F$.
 if $\prim{p}$ is inert or ramifies in $\gext$ then $\nu_{\prim{P}}(\alpha^{\sigma-1})=0$.
Assume $\prim{p}$ splits completely in $\gext$.
The $\prim{p}$-part of $\alpha$ may be written as $\prim{P}^{v_0}\,\prim{P}^{v_1\sigma}\,\prim{P}^{v_2\sigma^2}$
 with integers $v_0,\,v_1,\,v_2 \in \rint$.
Since $\prim{p} \in \primes_{F,S'}$ and the $\prim{p}$-part of $\alpha^{\sigma-1}$ is
 $\prim{P}^{v_2-v_0}\,\prim{P}^{(v_0-v_1)\sigma}\,\prim{P}^{(v_1-v_2)\sigma^2}$,
 by using Lemma~\ref{loceq} there are some integers $r,\,r' \in \rint$ such that $v_0-v_1=3r,\,v_1-v_2=3r'$.
It thus turns out that the $\prim{p}$-part of $\alpha^{\sigma-1}$ is equal to
 $(\prim{P}^{\sigma-1})^{3r}(\prim{P}^{\sigma^2-1})^{3r'}$.
Using the assumption $\iclsgp_{\gext,S'}^G=\iclsgp_{\gext,S'}$,
 there is some $\lambda_{\prim{p}} \in \gext^{*}$ such that
 $\prim{P}^{\sigma-1}\ringint{\gext,S'}=\lambda_{\prim{p}}\ringint{\gext,S'}$,
 and hence the $\prim{p}$-part is the principal ideal
 $\bigl(\lambda_{\prim{p}}^{r+r'(\sigma+1)}\bigr)^3\ringint{\gext,S'}$.
We can take such $\lambda_{\prim{p}} \in \gext^{*}$ for each $\prim{p} \in \primes_{F,S'}$
 so as to split completely in $\gext$, and let $\beta^{-1}$ be a product of all $\lambda_{\prim{p}}^{r+r'(\sigma+1)}$
 which varies through those $\prim{p}$,
 namely $\beta:=\prod \bigl(\lambda_{\prim{p}}^{r+r'(\sigma+1)}\bigr)^{-1} \in \gext^{*}$.
Here the integers $r,\,r'$ are both $0$ for almost all primes in $\primes_{F,S'}$.
Then $\alpha^{\sigma-1}\beta^3 \in \ringint{\gext,S'}^{*}$ by the definition.
This yields the lemma.
\end{proof}

We are now ready to prove the theorems.

\vsp

\begin{proof}[Proof of Theorem~\ref{bound}.]
\quad
This is a direct consequence deduced from Proposition~\ref{bound_R}, \ref{exact} and Lemma~\ref{kerG}, \ref{embed}.
\end{proof}

\vsp

\begin{proof}[Proof of Theorem~\ref{equality'}.]
\quad
Since $\Norm{\gext/F}\bigl(\gamma(\gp_H)\bigr)=-\delta(\gp)=\pm\delta(\mathfrak{g})$,
 under the assumption, a finite prime $\prim{p}$ of $F$ ramifies in $\gext$ if and only if
 $3 \nmid \nu_{\prim{p}}\bigl(\Norm{\gext/F}(\gamma(\gp_H))\bigr)$ by using Corollary~\ref{disc}.
From Proposition~\ref{bound_R} and the sequence \eqref{seq_F}, we have $S^{(\dual{\phi})}(\dual{\gp_H}/F)
={_\mathrm{N}\mathfrak{R}_{\gext}^G}={_\mathrm{N}\mathfrak{R}_{F,S}}={_\mathrm{N}\Phi_{F,S}}$.
Since $\zeta_3 \in F_{\prim{p}}$ for $\prim{p}$ which ramifies tamely in $\gext$,
 applying Proposition~\ref{duality} with the assumption $\delta(\gp) \notin F_{\prim{p}}^{*3}$ for $\prim{p} \mid 3$
 and Remark~\ref{val} yields $d_{+}=0,\,d_{-}-d_{\infty}=0$ and hence the equality
\begin{align}\label{SS}
\dim_{\mathbb{F}_3} S^{(\phi)}(\gp_H/F)+\rank_{\rint} \ringint{F}^{*}
=\dim_{\mathbb{F}_3} S^{(\dual{\phi})}(\dual{\gp_H}/F).
\end{align}
The estimate
 $1 \leq \rank_{\rint} \gp_H(F) \leq 2\dim_{\mathbb{F}_3} {_\mathrm{N}\Phi_{F,S}}-\rank_{\rint} \ringint{F}^{*}-1$
 now follows from the formula~\eqref{rank} with Lemma~\ref{MW}.
\end{proof}

\vsp

\begin{proof}[Proof of Theorem~\ref{equality}.]
\quad
The assumption is equivalent to that of Theorem~\ref{equality'} applied with $F=\rat$.
The first statement follows from Remark~\ref{vanish}, Lemma~\ref{kerG}, \ref{G-units}, the equality~\eqref{SS}
 and the formula \eqref{rank}.
Further assume $\iclsgp_{\gext}$ is $G$-invariant.
Then Lemma~\ref{kerG}, \ref{embed}, \ref{MW} imply
 $\gp_H(\rat)/\dual{\phi}\bigl(\dual{\gp_H}(\rat)\bigr) \simeq {_\mathrm{N}\mathfrak{U}_{\gext}} \simeq \Psi^G
 \simeq \rint/3\rint$.
Summarizing them with Proposition~\ref{exact} and Remark~\ref{vanish},
 we have the following equivalence of exact sequences.
$$
\begin{CD}
0 @>>>	{_\mathrm{N}\mathfrak{U}_{\gext}} @>>>	{_\mathrm{N}\mathfrak{R}_{\gext}^G} @>>>
 {_\mathrm{N}\iclsgp_{\gext}^G[3]} @>>> 0\\
@.	@|					@|			   @|					@.\\
0 @>>>	\gp_H(\rat)/\dual{\phi}\bigl(\dual{\gp_H}(\rat)\bigr) @>>> S^{(\dual{\phi})}(\dual{\gp_H}/\rat) @>>>
 \text{\cyr X}(\dual{\gp_H}/\rat)[\dual{\phi}]	@>>> 0.
\end{CD}
$$
This completes the proof.
\end{proof}

\section{Construction of the homomorphism $\dual{\delta}_{F_{\bullet}}^{\gext_{\bullet}}$}\label{sec_const}

This section is a key part to carry out our preceding descent procedure.
We continue to use the previous notation.
The idea is an interchange of the specified base point $\ellunit$ into another point on the elliptic curve $\gp_H$.
As a consequence, we obtain an available description for the image of the map
 $\dual{\delta}_{F_{\bullet}}^{\gext_{\bullet}}$.
We begin by taking such a point.
From the equation defining $\gp_H$, there are $\gext$-rational points on $\gp_H \cap \{w=0\}$ which consists of
\begin{align*}
&\ellunit'=[c\alpha_1-b\alpha_2,\,a\alpha_2-c,\,-a\alpha_1+b,\,0],\\
&T':=[c\alpha_1^{\sigma}-b\alpha_2^{\sigma},\,a\alpha_2^{\sigma}-c,\,-a\alpha_1^{\sigma}+b,\,0]
\;\bigl(=\ellunit'^{\sigma}\bigr),\\
&T'^{\sigma}\!\!=[c\alpha_1^{\sigma^2}-b\alpha_2^{\sigma^2},\,a\alpha_2^{\sigma^2}-c,\,-a\alpha_1^{\sigma^2}+b,\,0]
\;\bigl(=\ellunit'^{\sigma^2}\bigr).
\end{align*}
We adopt the point $\ellunit'$ as a new base point on the elliptic curve $\gp_H$ over $\gext$,
 and denote the new elliptic curve by $\gp_H':=(\gp_H,\,\ellunit')$.
Note that all the points $\ellunit',\,T',\,T'^{\sigma}$ lie on the same line $H \cap \{w=0\}$
 and hence $T'+'T'^{\sigma}=\ellunit'$ where the symbol $+'$ stands for the group law of $\gp_H'$.
The point $T'$ turns out to be a $3$-torsion element in the Mordell-Weil group $\gp_H'(\gext)$
 as a consequence of the following lemma.

\begin{lemma}\label{div}
\quad
A rational function $\eta \in \gext(\gp_H)^{*}$ defined by
$$\eta(x,\,y,\,z,\,w)=\frac{x+y\alpha_1^{\sigma}+z\alpha_2^{\sigma}}{x+y\alpha_1+z\alpha_2}$$
has the divisor $3(T')-3(\ellunit')$.
\end{lemma}

\vsp

\begin{proof}
\quad
It is easy to check that $x(P)+y(P)\alpha_1+z(P)\alpha_2=0$ if and only if $P=\ellunit'$.
From this, the divisor of the rational function $\eta$ has the form $n(T')-n(\ellunit')$
 with some positive integer $n$.
In the next place, we calculate the order of $\eta \in \gext(\gp_H)$ at the point $\ellunit'$.
The regular local ring $\gext[\gp_H]_{\ellunit'}$ at $\ellunit'$ is the localization of $\gext[\gp_H]$
 with respect to the maximal ideal
$$\text{\small $\bigl((a\alpha_2-c)x-(c\alpha_1-b\alpha_2)y,\,(-a\alpha_1+b)y
-(a\alpha_2-c)z,\,(c\alpha_1-b\alpha_2)z-(-a\alpha_1+b)x,\,w\bigr)$},$$
which turns out to be $(x+y\alpha_1+z\alpha_2,\,w)=(w)$ by using the defining equation of $\gp_H$.
Therefore the function $w$ defines a uniformizer in $\gext[\gp_H]_{\ellunit'}$.
This yields $\ord_{\ellunit'}(\eta)=-3$ by using the defining equation of $\gp_H$ once again.
Therefore $n$ must be equal to $3$ and we find $T'$ is a $3$-torsion point on $\gp_H'$
 by the well-known Abel-Jacobi theorem.
\end{proof}

Let $\tau:\gp_H \to \gp_H'$ be a translation map over $\gext$ defined by $P \mapsto P+\ellunit'$.
Then $\tau(T)=T'$ and the composed maps $\phi \circ \tau^{-1}:\gp_H' \to \dual{\gp_H}$,
 $\tau \circ \dual{\phi}:\dual{\gp_H} \to \gp_H'$ are $3$-isogenies defined over $\gext$
 since these maps send a base point to a base point.

The Weil pairing plays an important role
 for constructing a homomorphism from a Mordell-Weil group to a quotient multiplicative group of the ground field.
The constructed map is compatible with connecting homomorphisms of Galois cohomology
 and can be always described by a rational function $f_{P}$ on $\gp_H$
 if ever the support of the divisor of $f_P$ is $\{P,\,\ellunit'\}$
 where $P \ne \ellunit'$ is some torsion point defined over the ground field.
As for details, we refer to \cite{AEC}. We now apply these to our case.

By the notation $\gext_{\bullet}/F_{\bullet}$, we shall indicate the cyclic cubic extension $\gext/F$ or
 an extension $\gext_{\prim{P}}/F_{\prim{p}}$ of local fields.
The Weil pairing leads to the composed injection
$$\dual{\delta'}:\gp_H'(\gext_{\bullet})/\tau \circ \dual{\phi}\bigl(\dual{\gp_H}(\gext_{\bullet})\bigr)
 \longrightarrow H^1(\gext_{\bullet},\,\dual{\gp_H}[\dual{\phi}])
 \stackrel{e_{\tau \circ \dual{\phi}}(\,\cdot\,,\,T')}{\longrightarrow}
 H^1(\gext_{\bullet},\,\mu_3) \longrightarrow \gext_{\bullet}^{*}/\gext_{\bullet}^{*3}$$
which maps an element $P$ to $f_{T'}(P)\gext_{\bullet}^{*3}$
 for any $P \in \gp_H'(\gext_{\bullet}) \setminus \{\ellunit',\,T'\}$.
Here $f_{T'}$ is a rational function on $\gp_H'$ satisfying
$$\mathrm{div}(f_{T'})=3(T')-3(\ellunit'), \quad f_{T'} \circ \tau \circ \dual{\phi}=g_{T'}^3
 \text{ for some $g_{T'} \in \gext(\dual{\gp_H})$}.$$
By the Abel-Jacobi theorem, such functions $f_{T'} \in \gext(\gp_H'),\,g_{T'} \in \gext(\dual{\gp_H})$ always exist.

From Lemma~\ref{div}, since $\mathrm{div}(\eta)=3(T')-3(\ellunit')$,
 there is an element $\lambda \in \gext^{*}$ so that $f_{T'}=\lambda \cdot \eta$.
Thus the image of the composed map
 $\dual{\delta'} \circ \tau:\gp_H(\gext_{\bullet})/\dual{\phi}\bigl(\dual{\gp_H}(\gext_{\bullet})\bigr)
 \to \gext_{\bullet}^{*}/\gext_{\bullet}^{*3}$ is calculated as
\begin{align}
\begin{split}\label{eta}
\dual{\delta'} \circ \tau(P)
&=\dual{\delta'} \circ \tau(P-\ellunit')\bigl(\dual{\delta'} \circ \tau(-\ellunit')\bigr)^{-1}\\
&=\dual{\delta'}(P)\dual{\delta'}(\ellunit)^{-1}\\
&=\bigl(\lambda \cdot \eta(P)\bigr)\bigl(\lambda \cdot \eta(\ellunit)\bigr)^{-1}\gext_{\bullet}^{*3}\\
&=\eta(P)\gext_{\bullet}^{*3} \quad\quad \text{for any $P \in \gp_H(\gext_{\bullet}) \setminus \{\ellunit',\,T'\}$.}
\end{split}
\end{align}
It is easily seen that $\dual{\delta'} \circ \tau=\dual{\delta}_{\gext_{\bullet}}$
 by the equality $e_{\tau \circ \dual{\phi}}(P,\,T')=e_{\dual{\phi}}(P,\,T)$
 for any $P \in \dual{\gp_H}[\dual{\phi}]$.

Though we does not need to clarify the constant $\lambda \in \gext^{*}$ so far,
 it will be required to prove the proposition below.
We next give the description for $\lambda$.

\begin{lemma}\label{constant}
\quad
$\lambda \gext^{*3}=\frac{\gamma(\gp_H)}{\gamma(\gp_H)^{\sigma^2}} \gext^{*3}$.
\end{lemma}

\vsp

\begin{proof}
\quad
From the paper \cite{RK}, \S3,
 there is a homomorphism $\dual{\delta}_{\mathcal{W}}:\mathcal{W}(\gp_H)(\gext) \to \gext^{*}/\gext^{*3}$
 which sends the point $\vartheta(T)=(0,\,0)$ to $\Norm{\gext/F}\bigl(\gamma(\gp_H)\bigr)^{-1}\gext^{*3}$.
Taking account of the Weil pairing with $\tau(T)=T',\,\vartheta(T)=(0,\,0)$ yields
 $\dual{\delta'} \circ \tau=\dual{\delta}_{\mathcal{W}} \circ \vartheta$.
On the other hand, $\dual{\delta'} \circ \tau(T)=\dual{\delta'}(T')=\dual{\delta'}(T'^{\sigma})^{-1}
=\lambda^{-1} \eta(T'^{\sigma})^{-1} \gext^{*3}$.
Here we used the relation $T'+'T'^{\sigma}=\ellunit'$.
Moreover it is directly verified that $\eta(T'^{\sigma})=-\frac{\gamma(\gp_H)^{\sigma}}{\gamma(\gp_H)^{\sigma^2}}$.
We thus have $\Norm{\gext/F}\bigl(\gamma(\gp_H)\bigr)^{-1}\gext^{*3}=\lambda^{-1} \eta(T'^{\sigma})^{-1} \gext^{*3}$,
 namely $\lambda \gext^{*3}=\frac{\gamma(\gp_H)}{\gamma(\gp_H)^{\sigma^2}} \gext^{*3}$.
\end{proof}

Summarizing the above arguments yields the following proposition.

\begin{proposition}\label{connect}
\quad
The rational function $\eta \in \gext(\gp_H)^{*}$ in Lemma~\ref{div} describes the map
\begin{align*}
\dual{\delta}_{F_{\bullet}}^{\gext_{\bullet}}:
\gp_H(F_{\bullet})/\dual{\phi}\bigl(\dual{\gp_H}(F_{\bullet})\bigr) \quad &\longrightarrow
 \quad \gext_{\bullet}^{*}/\gext_{\bullet}^{*3}\\
[x,\,y,\,z,\,w] \quad &\longmapsto
 \quad \frac{x+y\alpha_1^{\sigma}+z\alpha_2^{\sigma}}{x+y\alpha_1+z\alpha_2} \gext_{\bullet}^{*3},\\
\ellunit' \quad &\longmapsto \quad \frac{\gamma(\gp_H)^{\sigma^2}}{\gamma(\gp_H)} \gext_{\bullet}^{*3},\\
T' \quad &\longmapsto \quad \frac{\gamma(\gp_H)}{\gamma(\gp_H)^{\sigma}} \gext_{\bullet}^{*3}.
\end{align*}
This homomorphism is injective if $\zeta_3 \notin F_{\bullet}$ or $F_{\bullet}=\gext_{\bullet}$.
\end{proposition}

\vsp

\begin{proof}
\quad
There is the commutative diagram
$$
\begin{CD}
\gp_H(F_{\bullet})/\dual{\phi}\bigl(\dual{\gp_H}(F_{\bullet})\bigr)		@>{\dual{\delta}_{F_{\bullet}}}>>
F_{\bullet}^{*}/F_{\bullet}^{*3}\\
@VVV								@VVV\\
\gp_H(\gext_{\bullet})/\dual{\phi}\bigl(\dual{\gp_H}(\gext_{\bullet})\bigr)	@>{\dual{\delta}_{\gext_{\bullet}}}>>
\gext_{\bullet}^{*}/\gext_{\bullet}^{*3}.
\end{CD}
$$
Here $\dual{\delta}_{F_{\bullet}},\,\dual{\delta}_{\gext_{\bullet}}$
 are injections and the vertical arrows are natural maps.
For any $P \in \gp_H(F_{\bullet}) \setminus \{\ellunit',\,T'\}$,
 we have $\dual{\delta}_{F_{\bullet}}^{\gext_{\bullet}}(P)=\dual{\delta}_{\gext_{\bullet}}(P)
=\dual{\delta'} \circ \tau(P)=\eta(P)\gext_{\bullet}^{*3}$ by \eqref{eta}.
If $P=\ellunit'$ then we can calculate as
 $\dual{\delta}_{F_{\bullet}}^{\gext_{\bullet}}(\ellunit')
=\bigl(\dual{\delta'} \circ \tau(-\ellunit')\bigr)^{-1}=\dual{\delta'}(\ellunit)^{-1}
=\lambda^{-1} \gext_{\bullet}^{*3}=\frac{\gamma(\gp_H)^{\sigma^2}}{\gamma(\gp_H)} \gext_{\bullet}^{*3}$
 by Lemma~\ref{constant}.
Similarly, we have $\dual{\delta}_{F_{\bullet}}^{\gext_{\bullet}}(T')
=\frac{\gamma(\gp_H)}{\gamma(\gp_H)^{\sigma}} \gext_{\bullet}^{*3}$.
The remaining statement follows from Lemma~\ref{kernel}.
\end{proof}

\section{Local calculations and some lemmas}

In this section we assemble some lemmas frequently used in the preceding sections.
 
\begin{lemma}\label{kernel}
\quad
We denote by $\gext_{\bullet}/F_{\bullet}$ either $\gext/F$ or $\gext_{\prim{P}}/F_{\prim{p}}$
 for finite primes $\prim{P}/\prim{p}$.
The natural homomorphism
 $F_{\bullet}^{*}/F_{\bullet}^{*3} \to \gext_{\bullet}^{*}/\gext_{\bullet}^{*3}$
 induced by the inclusion $F_{\bullet} \subset \gext_{\bullet}$ is injective if and only if
 $\zeta_3 \notin F_{\bullet}$ or $F_{\bullet}=\gext_{\bullet}$.
More precisely, for an extension $\gext_{\bullet}/F_{\bullet}$ the kernel of the natural homomorphism is
\begin{align*}
(F^{*} \cap \gext^{*3})/F^{*3}&=
\begin{cases}
\langle D F^{*3} \rangle \quad &\text{if $\zeta_3 \in F$},\\
\{F^{*3}\} \quad &\text{if $\zeta_3 \notin F$},
\end{cases}\\
(F_{\prim{p}}^{*} \cap \gext_{\prim{P}}^{*3})/F_{\prim{p}}^{*3}&=
\begin{cases}
\langle D_{\prim{p}} F_{\prim{p}}^{*3} \rangle
 \quad &\text{if $\zeta_3 \in F_{\prim{p}}$ and $\prim{p}$ does not split completely in $\gext$},\\
\{F_{\prim{p}}^{*3}\} \quad &\text{otherwise},
\end{cases}
\end{align*}
where $D_{\bullet}$ stands for some element in $F_{\bullet}^{*} \setminus F_{\bullet}^{*3}$
 satisfying $\gext_{\bullet}=F_{\bullet}(\sqrt[3]{D_{\bullet}})$.
Note that if $\zeta_3 \in F_{\prim{p}}$ and $F_{\prim{p}} \ne \gext_{\prim{P}}$ then
 $\langle D_{\prim{p}} F_{\prim{p}}^{*3} \rangle
=\Norm{\gext_{\prim{P}}/F_{\prim{p}}}(\gext_{\prim{P}}^{*})/F_{\prim{p}}^{*3}$.
Further if $\prim{p}$ ramifies tamely in $\gext$ then automatically $\zeta_3 \in F_{\prim{p}}$ and
 $\langle D_{\prim{p}} F_{\prim{p}}^{*3} \rangle=\langle \pi F_{\prim{p}}^{*3} \rangle$
 with some prime element $\pi \in F_{\prim{p}}$.
\end{lemma}

\vsp

\begin{proof}
\quad
There is equivalent exact sequences
$$
\begin{CD}
0	@>>>		(F_{\bullet}^{*} \cap \gext_{\bullet}^{*3})/F_{\bullet}^{*3}
	@>>>		F_{\bullet}^{*}/F_{\bullet}^{*3}
	@>>>		\gext_{\bullet}^{*}/\gext_{\bullet}^{*3}\\
@.			@|				@|				@|\\
0	@>>>		H^1(\gext_{\bullet}/F_{\bullet},\,\mu_3^{\Gal(\closure{F_{\bullet}}/\gext_{\bullet})})
	@>{\inf}>>	H^1(F_{\bullet},\,\mu_3)	@>{\res}>>	H^1(\gext_{\bullet},\,\mu_3).
\end{CD}
$$
Since the degree $[\gext_{\bullet}:F_{\bullet}]$ must divide $3$,
 the condition $\zeta_3 \notin F_{\bullet}$ is equivalent to $\zeta_3 \notin \gext_{\bullet}$.
Hence the cohomology group
 $H^1(\gext_{\bullet}/F_{\bullet},\,\mu_3^{\Gal(\closure{F_{\bullet}}/\gext_{\bullet})})$
 vanishes when $\zeta_3 \notin F_{\bullet}$ or $\prim{p}$ splits completely in $\gext$.
Otherwise $H^1(\gext_{\bullet}/F_{\bullet},\,\mu_3^{\Gal(\closure{F_{\bullet}}/\gext_{\bullet})})$
 is isomorphic to the group $\rint/3\rint$. This proves the first assertion.
The latter follows from theory of local fields and Kummer.
\end{proof}

\begin{lemma}\label{norm}
\quad
Let $\gext_{\bullet}/F_{\bullet}$ denote $\gext/F$ or $\gext_{\prim{P}}/F_{\prim{p}}$ for finite primes
 $\prim{P}/\prim{p}$ which does not split completely.
For an element $d \in F_{\bullet}^{*}$, the following conditions are equivalent.
\begin{list}{$\diamond$}{}
\item $d=\alpha^{\sigma-1}\beta^3$ with some $\alpha,\,\beta \in \gext_{\bullet}^{*}$ satisfying
 $\Norm{\gext_{\bullet}/F_{\bullet}}(\alpha) \in F_{\bullet}^{*3}$.
\item $d=\Norm{\gext_{\bullet}/F_{\bullet}}(\beta)$.
\end{list}
\end{lemma}

\vsp

\begin{proof}
\quad
Assume the first condition and let $\Norm{\gext_{\bullet}/F_{\bullet}}(\alpha)=d_0^3$
 for some $d_0 \in F_{\bullet}^{*}$ where $\Gal(\gext_{\bullet}/F_{\bullet}) \simeq \langle \sigma \rangle$.
Then $1=(\alpha^{\sigma-1})^{1-\sigma}\beta^{3(1-\sigma)}=r^3$.
Here $r:=d_0^{-1}\alpha^{\sigma}\beta^{1-\sigma} \in \mu_3(\gext_{\bullet})=\mu_3(F_{\bullet})$
 since the extension degree $[\gext_{\bullet}:F_{\bullet}]$ is equal to $3$.
We thus have $1=r^{\sigma^2-1}=\alpha^{1-\sigma}\beta^{-3}\Norm{\gext_{\bullet}/F_{\bullet}}(\beta)$,
 and hence $d=\Norm{\gext_{\bullet}/F_{\bullet}}(\beta)$.
The converse is immediate from
 $d=\Norm{\gext_{\bullet}/F_{\bullet}}(\beta)=(\beta^{1-\sigma^2})^{\sigma-1}\beta^3$.
\end{proof}

\begin{lemma}\label{G-inv}
\quad
$\bigl((\Ker \Norm{\gext/F})\gext^{*3}/\gext^{*3}\bigr)^G=\mu_3(F)\Norm{\gext/F}(\gext^{*})\gext^{*3}/\gext^{*3}$.
\end{lemma}

\vsp

\begin{proof}
\quad
Since each element in the right-hand side has a representative of the form
 $d=r\Norm{\gext/F}(\beta) \in F^{*}$ with $r \in \mu_3(F),\,\beta \in \gext^{*}$,
 the element $d\beta^{-3} \in \Ker \Norm{\gext/F}$ represents an element in the left-hand side.
Conversely, let $\alpha \in \Ker \Norm{\gext/F}$ represent an element in the left-hand side.
Then there is an element $\beta \in \gext^{*}$ such that $\alpha^{\sigma-1}=\beta^3$,
 which implies $\Norm{\gext/F}(\beta)=1$ from Lemma~\ref{norm}.
By using Hilbert's ``Satz $90$", the element $\beta$ is of the form $\beta_0^{1-\sigma}$
 with some $\beta_0 \in \gext^{*}$ and hence $(\alpha \beta_0^3)^{\sigma}=\alpha \beta_0^3 \in F^{*}$.
Letting $d:=\alpha \beta_0^3 \in F^{*}$, we have $d^3=\Norm{\gext/F}(\beta_0)^3$.
Therefore $\eqcls{\alpha}=\eqcls{d}$ lies in the right-hand side.
\end{proof}

\begin{lemma}\label{loc_ram}
\quad
For any finite prime $\prim{p}$ of $F$ which ramifies tamely in $\gext$,
The image of the map $\dual{\delta}_{F_{\prim{p}}}^{\gext_{\prim{P}}}$ is trivial.
Especially $\Image \dual{\delta}_{F_{\prim{p}}} \subset \langle \pi F_{\prim{p}}^{*3} \rangle$
 for some prime element $\pi \in F_{\prim{p}}$.
\end{lemma}

\vsp

\begin{proof}
\quad
Since $\gext_{\prim{P}}/F_{\prim{p}}$ is a proper extension,
 there is no $F_{\prim{p}}$-rational point on $\gp_H \cap \{w=0\}$.
For any point $P=[x,\,y,\,z,\,1] \in \gp_H(F_{\prim{p}})$ it is seen from the equation defining $\gp_H$ that
 $\nu_{\prim{P}}(x+y\alpha_1+z\alpha_2)=\nu_{\prim{P}}(x+y\alpha_1^{\sigma}+z\alpha_2^{\sigma})
=\nu_{\prim{P}}(x+y\alpha_1^{\sigma^2}+z\alpha_2^{\sigma^2})=0$,
 and hence $\eta(P) \equiv 1 \pmod{\prim{P}}$ where $\eta$ has defined in Lemma~\ref{div}.
Therefore $\eta(P)$ lies in $\gext_{\prim{P}}^{*3}$ by Hensel's Lemma, which implies the desired result
 from Proposition~\ref{connect} and Lemma~\ref{kernel}.
\end{proof}

\begin{lemma}\label{tame}
\quad
For arbitrary $\alpha \in \gext \setminus F$,
 if a finite prime $\prim{p}$ of $F$ ramifies tamely in $\gext$ then
 $3 \nmid \nu_{\prim{p}}\bigl(\Norm{\gext/F}(\alpha-\alpha^{\sigma})\bigr)$.
\end{lemma}

\vsp

\begin{proof}
\quad
Let $\gp$ be a cubic surface defined by the polynomial \eqref{defpoly}
 with the $F$-basis $\{1,\,\alpha,\,\alpha^2\}$ of $\gext$.
Consider the elliptic curve $\gp_H/F$ for the hyperplane $H:z=0$.
Then it is verified that
 $\Norm{\gext/F}\bigl(\gamma(\gp_H)\bigr)=-\delta(\gp)=-\Norm{\gext/F}(\alpha-\alpha^{\sigma})$.
If $3 \mid \nu_{\prim{p}}\bigl(\Norm{\gext/F}(\gamma(\gp_H))\bigr)$ then
 $\Image \dual{\delta}_{F_{\prim{p}}}=\ringint{F_{\prim{p}}}^{*}/\ringint{F_{\prim{p}}}^{*3}$
 by using \cite{RK}, Theorem~4.1,
 which is isomorphic to $\rint/3\rint$ since $\zeta_3 \in F_{\prim{p}}$ when $\prim{p}$ ramifies tamely.
However this contradicts Lemma~\ref{loc_ram}. Therefore we have the desired result.
\end{proof}

\begin{remark}
\quad
The statement of Lemma~\ref{tame} seems to separate from the arithmetic of elliptic curves.
Nevertheless it can be proven via certain elliptic curves.
\end{remark}

\begin{lemma}\label{coincide}
\quad
Let $\prim{p}$ be a finite prime of $F$ which ramifies in $\gext$.
Assume $3 \nmid \nu_{\prim{p}}\bigl(\Norm{\gext/F}(\gamma(\gp_H))\bigr)$.
If $\prim{p}$ ramifies wildly then we also assume $3$ splits completely in $F$. Then
 $\Image \dual{\delta}_{F_{\prim{p}}}=\Norm{\gext_{\prim{P}}/F_{\prim{p}}}(\gext_{\prim{P}}^{*})/F_{\prim{p}}^{*3}$.
\end{lemma}

\vsp

\begin{proof}
\quad
Let $\prim{P}$ be a prime above $\prim{p}$.
First of all, it is a direct consequence from Lemma~\ref{bound_norm} that $\Image \dual{\delta}_{F_{\prim{p}}}
 \hookrightarrow \Norm{\gext_{\prim{P}}/F_{\prim{p}}}(\gext_{\prim{P}}^{*})/F_{\prim{p}}^{*3}$.
It thus suffices to prove $\sharp \Image \dual{\delta}_{F_{\prim{p}}}
=\bigl[\Norm{\gext_{\prim{P}}/F_{\prim{p}}}(\gext_{\prim{P}}^{*}):F_{\prim{p}}^{*3}\bigr]$.
By using the equivalence $\Gal(\gext_{\prim{P}}/F_{\prim{p}}) \simeq
 F_{\prim{p}}^{*}/\Norm{\gext_{\prim{P}}/F_{\prim{p}}}(\gext_{\prim{P}}^{*})$ from local class field theory,
 it turns out that $\Norm{\gext_{\prim{P}}/F_{\prim{p}}}(\gext_{\prim{P}}^{*})/F_{\prim{p}}^{*3}
 \subsetneq F_{\prim{p}}^{*}/F_{\prim{p}}^{*3}$ whose group index is equal to $3$
 because $\prim{p}$ ramifies in $\gext$.
Under the assumption, since $[F_{\prim{p}}^{*}:F_{\prim{p}}^{*3}]=9$,
 it must be $\bigl[\Norm{\gext_{\prim{P}}/F_{\prim{p}}}(\gext_{\prim{P}}^{*}):F_{\prim{p}}^{*3}\bigr]=3$.
Applying \cite{RK}, Theorem~4.1 yields the required result.
\end{proof}

\begin{remark}\label{val}
\quad
Under the assumption of Lemma~\ref{coincide}, if $\prim{p}$ ramifies wildly in $\gext$ then 
 the case $\nu_{\prim{p}}\bigl(\Norm{\gext/F}(\gamma(\gp_H))\bigr) \equiv 1 \pmod{3}$
 cannot occur because this implies $\Image \dual{\delta}_{F_{\prim{p}}}=F_{\prim{p}}^{*}/F_{\prim{p}}^{*3}$
 by \cite{RK}, Theorem~4.1, which is contradiction.
\end{remark}

\begin{lemma}\label{G-units}
\quad
Assume $F=\rat$ then
 ${_\mathrm{N}\mathfrak{U}_{\gext}} \simeq (\ringint{\gext}^{*}/\ringint{\gext}^{*3})^G \simeq \rint/3\rint.$
\end{lemma}

\vsp

\begin{proof}
\quad
First, the inclusion ${_\mathrm{N}\mathfrak{U}_{\gext}} \hookrightarrow (\ringint{\gext}^{*}\gext^{*3}/\gext^{*3})^G
=(\ringint{\gext}^{*}/\ringint{\gext}^{*3})^G$ follows from
 $(\alpha^{\sigma-1})^{\sigma-1}=\Norm{\gext/\rat}(\alpha)\alpha^{-3\sigma} \in \gext^{*3}$
 for any $\eqcls{\alpha^{\sigma-1}} \in {_\mathrm{N}\mathfrak{U}_{\gext}}$.
Since $\gext$ is a totally real cubic field over $\rat$, the unit group is of the form
 $\ringint{\gext}^{*}=\langle -1,\,\epsilon_1,\,\epsilon_2 \rangle$ with fundamental units $\epsilon_1,\,\epsilon_2$,
 and then $\ringint{\gext}^{*}/\ringint{\gext}^{*3} \simeq (\rint/3\rint)^{\oplus 2}$.
It suffices to prove that there is some $\epsilon \in \ringint{\gext}^{*}$ so as to be
 $\eqcls{\epsilon} \notin (\ringint{\gext}^{*}/\ringint{\gext}^{*3})^G$.
If this is proven, then $\epsilon^{\sigma-1} \notin \ringint{\gext}^{*3}$, $\Norm{\gext/\rat}(\pm\epsilon)=1$
 and hence $\epsilon^{\sigma-1}$ represents a non-trivial element in ${_\mathrm{N}\mathfrak{U}_{\gext}}$
 because $(\epsilon^{\sigma-1})^{\sigma-1}=\Norm{\gext/\rat}(\epsilon)\epsilon^{-3\sigma} \in \ringint{\gext}^{*3}$.
Thus the lemma follows.
The units $\epsilon_1^{\sigma},\,\epsilon_2^{\sigma}$ can be written by the form
 $\pm \epsilon_1^r \epsilon_2^s,\,\pm \epsilon_1^{r'} \epsilon_2^{s'}$ with integers $r,\,s,\,r',\,s' \in \rint$,
 respectively.
Then the equalities $\epsilon_1^{\sigma^3}=\epsilon_1,\,\epsilon_2^{\sigma^3}=\epsilon_2$
 cause the following equations for $r,\,s,\,r',\,s'$.
\begin{align*}
r(r^2+r's)+sr'(r+s')=1, &\quad s(r^2+r's)+ss'(r+s')=0,\\
rr'(r+s')+r'(r's+s'^2)=0, &\quad sr'(r+s')+s'(r's+s'^2)=1.
\end{align*}
From this we obtain
$$(r,\,s,\,r',\,s')=
\begin{cases}
(r,\,-\frac{r^2+r+1}{r'},\,r',\,-1-r) &\text{if $ss' \ne 0$},\\
(-1,\,\pm1,\,-s,\,0) &\text{if $s \ne 0$ and $s'=0$},\\
(1,\,0,\,0,\,1) &\text{if $s=0$}.
\end{cases}$$
However the last case $s=0$ is impossible since $\epsilon_1,\,\epsilon_2 \in \gext \setminus F$.
Let $\epsilon:=\epsilon_1$ if $3 \nmid s$, while $\epsilon:=\epsilon_2$ if $3 \mid s$.
Then $\epsilon$ is the required element as follows.
Assume $3 \nmid s$ then $\epsilon_1^{\sigma-1} \notin \ringint{\gext}^{*3}$.
When $3 \mid s$ and $s \ne 0$, it must be $r \equiv 1 \pmod{3}$ and hence $9 \nmid r^2+r+1$.
This implies $3 \nmid r'$ because $s \in \rint$.
We thus have $\epsilon_2^{\sigma-1} \notin \ringint{\gext}^{*3}$.
\end{proof}

\begin{lemma}\label{MW}
\quad
Let $F$ be a number field of class number $1$ in which $3$ splits completely.
For a cyclic cubic extension $\gext/F$, let $\{1,\,\alpha,\,\alpha^2\}$ be a $F$-basis of $\gext$
 that $\gp$ is defined by the polynomial \eqref{defpoly} with.
Here $\alpha \in \gext$ denotes a root of the generic polynomial $\mathfrak{g}(x;t)$
 in \S\ref{genpoly} for some $t \in F$.
Let $H$ be the hyperplane $z=0$.
Assume $\delta(\mathfrak{g}) \notin F_{\prim{p}}^{*3}$  for each $\prim{p} \mid 3$. Then
$$\rint/3\rint \hookrightarrow \gp_H(F)/\dual{\phi}\bigl(\dual{\gp_H}(F)\bigr),
 \quad (\rint/3\rint)^{\oplus 2} \hookrightarrow \gp_H(F)/3\gp_H(F).$$
Especially $\rank_{\rint} \gp_H(F) \geq 1$ from the formula \eqref{rank}.
\end{lemma}

\vsp

\begin{proof}
\quad
By Remark~\ref{K}, the group $\dual{\gp_H}(F)/\phi\bigl(\gp_H(F)\bigr)$ contains $\rint/3\rint$ as subgroup.
It suffices to show that $\gp_H(F)/\dual{\phi}\bigl(\dual{\gp_H}(F)\bigr)$ also contains $\rint/3\rint$.
Then the statement follows from the formula
 $\gp_H(F)/3\gp_H(F) \simeq \gp_H(F)/\dual{\phi}\bigl(\dual{\gp_H}(F)\bigr) \oplus
 \dual{\gp_H}(F)/\phi\bigl(\gp_H(F)\bigr)$ in \S\ref{sec_descent}.
In the paper \cite{RK}, \S3, there is a homomorphism
 $\dual{\delta}_{\mathcal{W}}:\mathcal{W}(\gp_H)(F) \to F^{*}/F^{*3}$
 which sends the $3$-torsion point $(0,\,0)$ to
 $\Norm{\gext/F}\bigl(\gamma(\gp_H)\bigr)^{-1}F^{*3}=\delta(\mathfrak{g})^{-1}F^{*3}$.
We can easily verify the relation $\dual{\delta}_{F}=\dual{\delta}_{\mathcal{W}} \circ \vartheta$
 by $\vartheta(T)=(0,\,0)$.
Hence the subgroup generated by $\dual{\delta}_{F}(T)$ is equal to
 $\langle \delta(\mathfrak{g})F^{*3} \rangle$.
Since there is no proper unramified extension over $F$,
 $\delta(\mathfrak{g})$ has at least one prime divisor $\prim{p}$
 satisfying $3 \nmid \nu_{\prim{p}}\bigl(\delta(\mathfrak{g})\bigr)$ from Corollary~\ref{disc}.
Thus $\delta(\mathfrak{g})$ is not a perfect cube in $F$.
This implies $\eqcls{T} \in \gp_H(F)/\dual{\phi}\bigl(\dual{\gp_H}(F)\bigr)$ is a non-trivial element.
Therefore the proof is now ended.
\end{proof}

\section{A generic polynomial for cyclic cubic extensions}\label{genpoly}

For an explicit or systematic calculation, the notion of generic polynomials are often useful.
Generic polynomial is a polynomial which parameterizes all Galois extensions
 with a Galois group isomorphic to some fixed finite group $G$.
The existence of generic polynomials for a given finite group $G$ is unsolved problem.
But it is known in the case $G$ is cyclic of order $3$.

A generic polynomial for the cyclic group $C_3 \simeq \rint/3\rint$ over a number field $F$ used in the paper is
$$\mathfrak{g}(x;t):=x^3+t\,x^2-(t+3)x+1 \quad (t \in F)$$
whose discriminant is $(t^2+3t+9)^2$. Let $\delta(\mathfrak{g}):=t^2+3t+9$.
If $\alpha$ is a root of $\mathfrak{g}$ then $\frac{1}{1-\alpha},\,1-\frac{1}{\alpha}$ are all the other roots.
When $\alpha$ does not lie in $F$, the extension $K:=F(\alpha)/F$ is always cyclic cubic.
Then $\delta(\mathfrak{g})$ is also written as
$$\delta(\mathfrak{g})=\pm
\begin{vmatrix}
1 & \alpha & \alpha^2\\
1 & \alpha^{\sigma} & \alpha^{2\sigma}\\
1 & \alpha^{\sigma^2} & \alpha^{2\sigma^2}
\end{vmatrix}
\;\bigl(=\pm\Norm{\gext/F}(\alpha-\alpha^{\sigma})\bigr).
$$
Here $\Gal(\gext/F)=\langle \sigma \rangle \simeq C_3$
 and the sign $\pm$ depends on the choice of a generator $\sigma$.
From the property, for arbitrary cyclic cubic extension $\gext$ over $F$,
 there is some $t \in F$ so that the minimal splitting field of $\mathfrak{g}(x;t) \in F[x]$ is exactly $\gext$.

As for the conductor of a cyclic cubic extension $\gext/F$, we have the following proposition.

\begin{proposition}\label{conductor}
\quad
Assume $3$ splits completely in $F$.
For any finite prime $\prim{p}$ of $F$,
 the $\prim{p}$-conductor $\mathfrak{f}_{\prim{p}}=\prim{p}^{n_{\prim{p}}}$ of the cubic extension $\gext/F$
 is given by the following chart.
\begin{align*}
\nu_{\prim{p}}\bigl(\delta(\mathfrak{g})\bigr) \leq 0 &\rightarrow n_{\prim{p}}=0,\\
\nu_{\prim{p}}\bigl(\delta(\mathfrak{g})\bigr)>0 &\rightarrow
\begin{cases}
3 \nmid \nu_{\prim{p}}\bigl(\delta(\mathfrak{g})\bigr) \rightarrow n_{\prim{p}}=
\begin{cases}
1 &\text{\rm if $\prim{p} \nmid 3$},\\
2 &\text{\rm if $\prim{p} \mid 3$},
\end{cases}\\
3 \mid \nu_{\prim{p}}\bigl(\delta(\mathfrak{g})\bigr) \rightarrow
\begin{cases}
\prim{p} \nmid 3 \rightarrow n_{\prim{p}}=0,\\
\prim{p} \mid 3 \rightarrow n_{\prim{p}}=
\begin{cases}
0 &\text{\rm if $\delta(\mathfrak{g}) \notin F_{\prim{p}}^{*3}$},\\
2 &\text{\rm if $\delta(\mathfrak{g}) \in F_{\prim{p}}^{*3}$}.
\end{cases}
\end{cases}
\end{cases}
\end{align*}
\end{proposition}

\vsp

\begin{proof}
\quad
The proof for the case $\nu_{\prim{p}}\bigl(\delta(\mathfrak{g})\bigr) \geq 0$,
 but except for the case $3 \mid \nu_{\prim{p}}\bigl(\delta(\mathfrak{g})\bigr)>0$ with $\prim{p} \mid 3$,
 can be done essentially by the same way as Washington's one in \cite{Was}, which is hence omitted.
We assume $-v:=\nu_{\prim{p}}\bigl(\delta(\mathfrak{g})\bigr)<0$ and let $\pi$ be a prime element in $F_{\prim{p}}$.
Then the polynomial $\pi^{3v}\mathfrak{g}(\pi^{-v}x;t)$ has integral coefficients
 and the polynomial modulo $\pi$ has a simple root.
Therefore the root lifts to some $\alpha \in \ringint{F_{\prim{p}}}$
 such that $\pi^{3v}\mathfrak{g}(\pi^{-v}\alpha;t)=0$ by Hensel's lemma.
It thus turns out that all the roots of $\mathfrak{g}(x;t)$ are $\bigl\{\pi^{-v}\alpha,\,\frac{1}{1-\pi^{-r}\alpha},
\,1-\frac{1}{\pi^{-v}\alpha}\bigr\} \subset F_{\prim{p}}$.
This implies that $\prim{p}$ splits completely.
Finally in the case $3 \mid \nu_{\prim{p}}\bigl(\delta(\mathfrak{g})\bigr)>0$ with $\prim{p} \mid 3$,
 we must show that $\prim{p}$ ramifies wildly if and only if $\delta(\mathfrak{g}) \in F_{\prim{p}}^{*3}$.
First, $t$ must be of the form $9t'+3$ with some $t' \in \ringint{F_{\prim{p}}}$
 by $3 \mid \nu_{\prim{p}}\bigl(\delta(\mathfrak{g})\bigr)>0$.
If $\delta(\mathfrak{g}) \notin F_{\prim{p}}^{*3}$ then $t' \equiv 1 \pmod{3}$
 and hence the polynomial $3^{-3}\mathfrak{g}(3x-\frac{t}{3};t)$ has coefficients in $\ringint{F_{\prim{p}}}$
 and discriminant not divisible by $\prim{p}$. Therefore $\prim{p}$ is unramified in $\gext$.
If $\delta(\mathfrak{g}) \in F_{\prim{p}}^{*3}$ then $t' \not\equiv 1 \pmod{3}$
 and hence the polynomial $\frac{3^3x^3}{(2t+3)\delta(\mathfrak{g})}\mathfrak{g}(\frac{3}{x}-\frac{t}{3};t)$
 turns out to be an Eisenstein polynomial for $\prim{p}$. Thus $\prim{p}$ ramifies in $\gext$.
It is a fact that $n_{\prim{p}} \leq 2$ for any $\prim{p}$ when $3$ is unramified in $F$.
Therefore $n_{\prim{p}}$ is equal to either $1$ or $2$ according to $\prim{p}$ ramifies tamely or wildly,
 respectively.
\end{proof}

The following is immediate from the above proposition.

\begin{corollary}\label{disc}
\quad
Assume $3$ splits completely in $F$ and $\delta(\mathfrak{g}) \notin F_{\prim{p}}^{*3}$ for each $\prim{p} \mid 3$.
Then a finite prime $\prim{p}$ of $F$ ramifies in $\gext$ if and only if
 $3 \nmid \nu_{\prim{p}}\bigl(\delta(\mathfrak{g})\bigr)>0$.
\end{corollary}

\section*{Acknowledgements}

I would like to thank Professor Masanobu Kaneko for helpful suggestions and supports.

\end{document}